\documentclass[a4paper,letterpaper,leqno]{amsart}
\usepackage{amsmath, amsthm, amsfonts}
\usepackage{amssymb}
\usepackage{latexsym}
\usepackage{verbatim}
\usepackage{url}
\usepackage{pdflscape,multirow,array,bigstrut}
\usepackage{fullpage}
\usepackage[all,cmtip]{xy}
\usepackage{color}
\usepackage[usenames,dvipsnames]{xcolor}
\usepackage{hyperref}
\hypersetup{colorlinks=true,urlcolor=MidnightBlue,citecolor=PineGreen,linkcolor=BrickRed}

\usepackage{enumerate}
\usepackage{amssymb}
\usepackage{array}
\usepackage{nicefrac}
\usepackage[section]{placeins}

\DeclareMathAlphabet{\mathfr}{U}{euf}{m}{n}

\newtheorem{theorem}{Theorem}[section]
\newtheorem{conjecture}[theorem]{Conjecture}
\newtheorem{proposition}[theorem]{Proposition}
\newtheorem{corollary}[theorem]{Corollary}
\newtheorem{hypothesis}[theorem]{Hypothesis}
\newtheorem{lemma}[theorem]{Lemma}

\theoremstyle{definition}

\newtheorem{remark}[theorem]{Remark}

\numberwithin{equation}{section}

\newcommand{\C}{\mathbb C}

\newcommand{\Q}{\mathbb Q}
\newcommand{\Qbar}{{\overline{\mathbb Q}}}
\newcommand{\R}{\mathbb R}
\newcommand{\Z}{\mathbb Z}

\newcommand{\fa}{\mathfrak a}
\newcommand{\g}{\mathfrak g}
\newcommand{\h}{\mathfrak h}
\newcommand{\p}{\mathfrak p}
\newcommand{\PP}{\mathfrak P}
\newcommand{\s}{\mathfrak s}

\newcommand{\USp}{\mathrm{USp}}
\newcommand{\Frob}{\mathrm{Frob}}

\newcommand{\W}{\mathcal W}
\newcommand{\Cc}{\mathcal C}
\newcommand{\Rr}{\mathcal R}
\newcommand{\Aa}{\mathcal A}
\newcommand{\Uu}{\mathcal U}
\newcommand{\Vv}{\mathcal V}
\newcommand{\Zz}{\mathcal Z}
\newcommand{\Ss}{\mathcal S}

\newcommand{\mb}{\boldsymbol m}

\newcommand{\tb}{\boldsymbol \theta}
\newcommand{\ttb}{\boldsymbol t}
\newcommand{\zb}{\boldsymbol z}
\newcommand{\nb}{\boldsymbol n}
\newcommand{\db}{\boldsymbol \delta}
\newcommand{\ab}{\boldsymbol a}
\newcommand{\vb}{\boldsymbol v}
\newcommand{\ub}{\boldsymbol u}

\newcommand{\oub}{\boldsymbol 0}
\newcommand{\ssb}{\boldsymbol s}
\newcommand{\vartb}{\boldsymbol \vartheta}

\newcommand{\Conj}{\operatorname{Conj}}
\newcommand{\Sp}{\operatorname{Sp}}
\newcommand{\GSp}{\operatorname{GSp}}
\newcommand{\Tr}{\operatorname{Trace}}
\newcommand{\End}{\operatorname{End}}
\newcommand{\Aut}{\operatorname{Aut}}
\newcommand{\Li}{\operatorname{Li}}
\newcommand{\ST}{\operatorname{ST}}

\newcommand{\NN}{\operatorname{Nm}}
\newcommand{\Unitary}{\operatorname{U}}
\newcommand{\Hom}{\operatorname{Hom}}
\newcommand{\GL}{\operatorname{GL}}

\newcommand{\MT}{\operatorname{MT}}

\newcommand{\fund}{\operatorname{fund}}
\newcommand{\disc}{\operatorname{disc}}
\newcommand{\Zar}{\operatorname{Zar}}
\newcommand{\AST}{\operatorname{AST}}
\newcommand{\Hg}{\operatorname{Hg}}
\newcommand{\diag}{\operatorname{diag}}

\definecolor{myorange}{rgb}{0.7,0.3,0}
\definecolor{myblue}{rgb}{.2,.6,.75}
\definecolor{mygreen}{rgb}{.4,.7,.4}

\newcommand{\leaveout}[1]{}

\begin{document}
\title{Effective Sato--Tate conjecture for abelian varieties\\ and applications}

\author{Alina Bucur}
\address{Department of Mathematics \\ University of California, San Diego \\ 9500 Gilman Drive \#0112 \\ 
La Jolla \\ CA 92093 \\ USA}
\email{alina@math.ucsd.edu}
\urladdr{https://www.math.ucsd.edu/~alina/}

\author{Francesc Fit\'e}
\address{Departament de matem\`atiques i inform\`atica,
Universitat de Barce\-lona,
Gran via de les Corts Catalanes 585, 08007 Barcelona, Catalonia, Spain}
\email{ffite@ub.edu}
\urladdr{http://www.ub.edu/nt/ffite/}

\author{Kiran S. Kedlaya}
\address{Department of Mathematics \\ University of California, San Diego \\ 9500 Gilman Drive \#0112 \\ 
La Jolla \\ CA 92093 \\ USA}
\email{kedlaya@ucsd.edu}
\urladdr{http://kskedlaya.org}

\date{\today}

\begin{abstract}
From the generalized Riemann hypothesis for motivic $L$-functions, we derive an effective version of the Sato--Tate conjecture for an abelian variety $A$ defined over a number field $k$ with connected Sato--Tate group. By \emph{effective} we mean that we give an upper bound on the error term in the count predicted by the Sato--Tate measure that only depends on certain invariants of~$A$. We discuss three applications of this conditional result. First, for an abelian variety defined over $k$, we consider a variant of Linnik's problem for abelian varieties that asks for an upper bound on the least norm of a prime whose normalized Frobenius trace lies in a given interval. Second, for an elliptic curve defined over~$k$ with complex multiplication, we determine (up to multiplication by a nonzero constant) the asymptotic number of primes whose Frobenius trace attain the integral part of the Hasse--Weil bound. Third, for a pair of abelian varieties $A$ and $A'$ defined over $k$ with no common factors up to $k$-isogeny, we find an upper bound on the least norm of a prime at which the respective Frobenius traces of $A$ and $A'$ have opposite sign.
\end{abstract}

\maketitle

\tableofcontents

\section{Introduction}\label{section: introduction}

Let $A$ be an abelian variety defined over a number field $k$ of dimension $g\geq 1$. For a rational prime $\ell$, we denote by 
$$
\varrho_{A,\ell}\colon G_k \rightarrow \Aut(V_\ell(A))
$$
the $\ell$-adic representation attached to $A$, obtained from the action of the absolute Galois group of $k$ on the rational $\ell$-adic Tate module $V_\ell(A):=T_\ell(A)\otimes \Q_\ell$. Let $N$ denote the absolute norm of the conductor of~$A$, which we will call the absolute conductor of~$A$. For a nonzero prime ideal $\p$ of the ring of integers of $k$ not dividing~$N\ell$, let  $a_\p:=a_\p(A)$ denote the trace of $\varrho_{A,\ell}(\Frob_\p)$, where $\Frob_\p$ is a Frobenius element at $\p$. The trace $a_\p$ is an integer which does not depend on $\ell$ and, denoting by $\NN(\p)$ the absolute norm of $\p$, the Hasse-Weil bound asserts that the normalized trace 
$$
\overline a_\p:=\frac{a_\p}{\sqrt{\NN(\p)}}
$$ 
lies in the interval $[-2g,2g]$. 

Attached to $A$ there is a compact real Lie subgroup $\ST(A)$ of the unitary symplectic group $\USp(2g)$ that conjecturally governs the distribution of the  normalized Frobenius traces. More precisely, the Sato--Tate conjecture predicts that the sequence $\{\overline a_\p\}_\p$, indexed by primes $\p$ not dividing $N$ ordered by norm, is equidistributed on the interval $[-2g,2g]$ with respect to the pushforward via the trace map of the (normalized) Haar measure of the Sato--Tate group $\ST(A)$. We will denote this measure by $\mu$.

Denote by $\delta_I$ the characteristic function of a subinterval $I$ of $[-2g,2g]$. Together with the prime number theorem, the Sato--Tate conjecture predicts that
\begin{equation}\label{equation: STPN}
\sum_{\NN(\p)\leq x}\delta_I(\overline a_\p)\sim \mu(I)\Li(x) \qquad\text{as }x\rightarrow \infty\,,
\end{equation}
where $\Li(x):=\int_2^\infty dt/\log(t)$. Let $L(\chi,s)$ denote the (normalized) $L$-function attached to an irreducible character $\chi$ of $\ST(A)$. It is well known that \eqref{equation: STPN} is implied by the conjectural nonvanishing and analyticity on the right halfplane $\Re(s)\geq 1$ of $L(\chi,s)$ for every nontrivial irreducible character $\chi$. In this paper we derive an asymptotic upper bound on the error term implicit in~\eqref{equation: STPN} by further assuming the generalized Riemann hypothesis for the $L$-functions $L(\chi,s)$. 
 
Our main result is a quantitative refinement of the Sato--Tate conjecture (see Theorem \ref{theorem: effectivest}). In order to state it we need to introduce some notations. Let $\g$ denote the complexified Lie algebra of $\ST(A)$, and write it as $\s\times \fa$, where $\s$ is semisimple and~$\fa$ is abelian. Set
\begin{equation}\label{equation: epsilondefinition}
\varepsilon_\g:=\frac{1}{2(q+\varphi)}\,,
\end{equation}
where $\varphi$ is the size of the set of positive roots of $\s$ and $q$ is the rank of $\g$, and define
\begin{equation}\label{equation: nuinterval}
\nu_\g\colon \R_{>0}\rightarrow \R_{>0}\,,\qquad \nu_{\g}(z)=\max\left\{1, \frac{\log(z)^6}{z^{1/\varepsilon_\g}}\right\}
\end{equation}
For a subinterval $I$ of $[-2g,2g]$, let $|I|$ denote its length.
\begin{theorem}[Effective Sato--Tate conjecture]\label{theorem: intro1}
Let~$A$ be an abelian variety defined over the number field~$k$ of dimension $g\geq 1$, absolute conductor $N$, and such that $\ST(A)$ is connected. Suppose that the Mumford--Tate conjecture holds for $A$ and that the generalized Riemann hypothesis holds for $L(\chi,s)$ for every irreducible character~$\chi$ of $\ST(A)$. 
Then for all subintervals $I$ of $[-2g,2g]$ of nonzero length, we have
\begin{equation}\label{equation: main count}
\sum_{\NN(\p)\leq x}\delta_{I}(\overline a_{\p})=\mu(I)\Li(x)+O\left(\frac{x^{1-\varepsilon_\g}\log(Nx)^{2\varepsilon_\g}}{\log (x)^{1-4\varepsilon_\g}}\right)\quad \text{for }x\geq x_0\,,
\end{equation}
where the sum runs over primes not dividing $N$, the implied constant in the $O$-notation depends exclusively on $k$ and $g$, and $x_0=O\left( \nu_\g(|I|)\log(2N)^2\log(\log(4N))^4 \right)$. 
\end{theorem}

The dependence on $g$ in the implied constant of \eqref{equation: main count} can be traced through Propositions \ref{proposition: Gupta} and \ref{proposition: dimbound} and Lemmas \ref{lemma: Cartanmult} and \ref{lemma: gradcontrol}; it is highly exponential.
This theorem generalizes a result of Murty \cite{Mur85} concerning elliptic curves without complex multiplication (CM); see also \cite[Thm. 3.1]{BK16b}. Its proof follows the strategy envisaged in \cite[\S5]{BK16b} and it occupies~\S\ref{section: efstc}. A key ingredient is the construction of a multivariate Vinogradov function; this is a continuous periodic function, with rapidly decaying Fourier coefficients, and approximating the characteristic function of the preimage of~$I$ by the trace map in the parameter space of a Cartan subgroup $H$ of $\ST(A)$. By identifying the quotient of this space by the action of the Weyl group with the set of conjugacy classes of $\ST(A)$, one can rewrite (a Weyl average of) the Vinogradov function as a combination of irreducible characters of $\ST(A)$. One can use purely Lie algebra theoretic arguments (most notably Weyl's character dimension formula and a result due to Gupta \cite[Thm. 3.8]{Gup87} on the boundedness of the inverse of the weight multiplicity matrix) to show that the coefficients in the character decomposition of the Vinogradov function also exhibit a rapid decay. The theorem can then be obtained by using an estimate of Murty (as presented in \cite[(2.4)]{BK16b}) on truncated sums of an irreducible character $\chi$ over the prime ideals of $k$. The implied constant in the $O$-notation depends in principle on the exponents of the Cartan subgroup $H$. In order to bound these exponents purely in terms of~$g$, we show that the Mumford--Tate conjecture implies that $H$ is generated by the Hodge circles contained in it (see Theorem \ref{theorem: Hodge}). This result may be of independent interest. 

The conjectural background for Theorem \ref{theorem: intro1} is presented in \S\ref{section: conjectures}. We recall the Mumford--Tate conjecture and the related algebraic Sato--Tate conjecture, define the $L$-functions $L(\chi,s)$, and state the generalized Riemann hypothesis for them. 
In \S\ref{section: Applications} we give three applications of Theorem \ref{theorem: intro1}. The first is what we call the \emph{interval variant} of Linnik's problem for an abelian variety (see Corollary \ref{corollary: intlin}).

\begin{corollary}\label{corollary: intro2}
Assume the hypotheses and notations of Theorem \ref{theorem: intro1}. For every subinterval $I$ of $[-2g,2g]$ of nonzero length, there exists a prime $\p$ not dividing $N$ with norm
$$
\NN(\p)=O\left(\nu_\g(\min\{|I|,\mu(I)\})\log(2N)^2\log(\log(4N))^4\right)
$$
such that $\overline a_\p\in I$. 
\end{corollary}

The second application concerns what we call the \emph{Frobenius sign separation problem} for a pair of abelian varieties (see Corollary~\ref{corollary: linniksign}). 

\begin{corollary}\label{corollary: intro3}
Let~$A$ (resp. $A'$) be an abelian variety defined over the number field~$k$ of dimension $g\geq 1$ (resp. $g'\geq 1$), absolute conductor $N$ (resp. $N'$), and such that $\ST(A)$ (resp. $\ST(A')$) is connected.  Suppose that the Mumford--Tate conjecture holds for $A$ (resp. $A'$) and that the generalized Riemann hypothesis holds for $L(\chi,s)$ (resp. $L(\chi',s)$) for every irreducible character~$\chi$ of $\ST(A)$ ($\chi'$ of $\ST(A')$). Suppose that $\ST(A\times A')\simeq \ST(A)\times \ST(A')$. Then there exists a prime~$\p$ not dividing $NN'$ with norm
$$
\NN(\p)=O\left(\log(2NN')^2\log(\log(4NN'))^6 \right)
$$ 
such that $a_\p(A)$ and $a_\p(A')$ are nonzero and of opposite sign. Here, the implied constant in the $O$-notation depends exclusively on $k$, $g$, and $g'$.  
\end{corollary}

We also examine what our method says about the set of primes with ``maximal Frobenius trace''. Let $M_k(x)$ denote the set of primes $\p$ not dividing $N$ with norm up to $x$ for which $a_\p=\lfloor 2\sqrt{\NN(\p)}\rfloor$. Vaguely formulated, a natural approach to compute (at least an asymptotic lower bound on) $M_k(x)$ is to compute the number of $\p$ with norm up to $x$ for which~$\overline a_\p$ lies in a sufficiently small neighborhood~$I_x$ of~$2g$. However, for this idea to succeed, the neighborhood~$I_x$ should be sufficiently large in order for the ``error term" in \eqref{equation: main count} to be still dominated by the ``main term", which is now multiplied by the tiny quantity~$\mu(I_x)$. In the case where $A$ is an elliptic curve with CM it is possible to achieve this trade-off, yielding the following statement (see Proposition~\ref{proposition: asympmax} and Corollary~\ref{corollary: asympmax}).

\begin{corollary}\label{corollary: intro4}
Let $A$ be an elliptic curve defined over $k$ with potential CM, that is, such that $A_\Qbar$ has CM. Under the generalized Riemann hypothesis for the $L$-function attached to every power of the Hecke character of~$A$, we have
$$
\#M_k(x)\asymp  \frac{x^{3/4}}{\log(x)}\qquad \text{as }x\rightarrow \infty\,.
$$
\end{corollary}

This recovers a weaker version of a theorem of James and Pollack \cite[Theorem~1]{JP17}, which asserts (unconditionally) that
$$
\#M_k(x) \sim \frac{2}{3\pi} \frac{x^{3/4}}{\log(x)}.
$$
A different result in a similar spirit, concerning numbers of points on diagonal curves, is due to Duke \cite[Theorem~3.3]{Duk89}.

Corollary \ref{corollary: intro3} extends work of Bucur and Kedlaya \cite[Thm. 4.3]{BK16b}, who considered the case in which~$A$ and~$A'$ are elliptic curves without CM. Later Chen, Park, and Swaminathan \cite[Thm. 1.3]{CPS18} reexamined this case, obtaining an upper bound of the form $O(\log(NN')^2)$ and relaxing the generalized Riemann hypothesis assumed in \cite[Thm. 4.3]{BK16b}. Corollary \ref{corollary: intro2} extends \cite[Thm. 1.8]{CPS18}, that again applies to elliptic curves without CM. 
It should be noted that the aforementioned results in \cite{CPS18} make explicit the constants involved in the respective upper bounds, a goal which we have not pursued in our work. 

The framework of the generalized Sato-Tate conjecture includes many additional questions about distinguishing $L$-functions, a number of which have been considered previously.
For instance, Goldfeld and Hoffstein \cite{goldfeld-hoffstein} established an upper bound on the first distinguishing coefficient for a pair of holomorphic Hecke newforms, by an argument similar to ours but with a milder analytic hypothesis (the Riemann hypothesis for the Rankin-Selberg convolutions of the two forms with themselves and each other).
Sengupta \cite{sengupta} carried out the analogous analysis with the Fourier coefficients
replaced by normalized Hecke eigenvalues (this only makes a difference when the weights are distinct). 

There is an alternative approach to the above kind of questions, which is based on the use of effective forms of Chebotaryov's density theorem conditional to the the Riemann hypothesis for Artin $L$-functions. This approach was introduced by Serre \cite{Ser81}, who gave an upper bound on the smallest prime at which two nonisogenous elliptic curves have different Frobenius traces. The analogue of Serre's argument for modular forms was given by Ram Murty
\cite{ram} and subsequently extended to Siegel modular forms by Ghitza \cite{ghitza} for Fourier coefficients and Ghitza and Sayer \cite{ghitza-sayer} for Hecke eigenvalues. Building on Serre's method, several recent works have explored the asymptotic number of zero Frobenius traces for abelian varieties which are either generic (see \cite{CW22b}) or isogenous to a product of elliptic curves (see \cite{HJS22} and \cite{CW22a}).

\subsection*{Notation and terminology} Throughout this article, $k$ is a fixed number field and $g$ and $\,g'$ are fixed positive integers. For an ordered set $(X,\leq)$ and functions $f,h\colon X\rightarrow \R$ we write $f(x)=O(h(x))$ to denote that there exist a real number $K>0$ and an element $x_0 \in X$ such that $|f(x)|\leq Kh(x)$ for every $x\geq x_0$. We will generally specify the element $x_0$ in the statements of theorems, but we will usually obviate it in their proofs, where it can be inferred from the context. We refer to $K$ as the \emph{implied constant} in the $O$-notation. As we did in this introduction, whenever using the $O$-notation in a statement concerning an arbitrary abelian variety $A$ of dimension $g$ defined over the number field $k$, the corresponding implied constant is computable exclusively in terms of $g$ and $k$ (in fact the dependence on~$k$ is just on the absolute discriminant $|\disc_{k/\Q}|$ and the degree $[k:\Q]$). For statements concerning a pair of arbitrary abelian varieties $A$ and $A'$ of respective dimensions~$g$ and~$g'$ defined over~$k$, the implied constant in the $O$-notation is computable exclusively in terms of $g$, $g'$, and~$k$.
Section \ref{section: maxnumpoints} is the only exception to the previous convention and to emphasize the dependency on~$N$ of the implied constants in the asymptotic bounds therein, we use the notations~$O_N$ and~$\asymp_N$. We write $f \asymp g$ if $f=O(g)$ and $g=O(f)$. By a prime of $k$, we refer to a nonzero prime ideal of the ring of integers of $k$.
Additional notation introduced later in the paper is summarized in Table~\ref{table:notations}.

\subsection*{Acknowledgements} We thank Christophe Ritzenthaler for raising the question of the infiniteness of the set of primes at which the Frobenius trace attains the integral part of the Weil bound, and Jeff Achter for directing us to \cite{JP17}. This occurred during the AGCCT conference held at CIRM, Luminy, in June 2019, where Bucur gave a talk based on this article; we also thank the organizers for their kind invitation. We thank Andrew Sutherland for providing the example of Remark~\ref{remark: dimensionsST} and numerical data compatible with Proposition \ref{proposition: asympmax}. We thank Jean-Pierre Serre for sharing the preprint of \cite{Ser20} with us and for precisions on a previous version of this manuscript. 

All three authors were supported by the Institute for Advanced Study during 2018--2019; this includes funding from National Science Foundation grant DMS-1638352. All three authors were additionally supported by the Simons Foundation grant 550033. Bucur was also supported by the Simons Foundation collaboration grant 524015, and by NSF grants DMS-2002716 and DMS-2012061. Kedlaya was additionally supported by NSF grants DMS-1501214, DMS-1802161, DMS-2053473 and by the UCSD Warschawski Professorship. Fit\'e was additionally supported by the Ram\'on y Cajal fellowship RYC-2019-027378-I, by the Mar\'ia de Maeztu program CEX2020-001084-M, by the DGICYT grant MTM2015-63829-P, and by the ERC grant 682152.

\section{Conjectural framework}\label{section: conjectures}

Throughout this section $A$ will denote an abelian variety of dimension~$g$ defined over the number field~$k$, and of absolute conductor $N:=N_A$. We will define its Sato--Tate group, introduce the motivic $L$-functions attached to it, and present the conjectural framework on which \S\ref{section: efstc} is sustained.

\subsection{Sato--Tate groups}
\label{subsec: Sato-Tate groups}

Following \cite[Chap. 8]{Ser12} (see also \cite[\S2]{FKRS12}), one defines the Sato--Tate group of $A$, denoted $\ST(A)$, in the following manner. Let $G_\ell^{\Zar}$ denote the Zariski closure of the image of the $\ell$-adic representation $\varrho_{A,\ell}$, which we may naturally see as lying in $\GSp_{2g}(\Q_\ell)$. Denote by $G_\ell^{1,\Zar}$ the intersection of  
$G_\ell^{\Zar}$ with $\Sp_{2g}/\Q_\ell$. Fix an isomorphism $\iota\colon \Qbar_\ell\simeq \C$ and let $G^{1,\Zar}_{\ell,\iota}$ denote the base change $G^{1,\Zar}_\ell\times_{\Q_\ell,\iota} \C$. The Sato--Tate group $\ST(A)$ is defined to be a maximal compact subgroup of the group of $\C$-points of $G^{1,\Zar}_{\ell,\iota}$. In the present paper, to avoid the a priori dependence on $\ell$ and $\iota$ of the definition of $\ST(A)$, we formulate the following conjecture.
\begin{conjecture}[Algebraic Sato--Tate conjecture; \cite{BK15a}]
There exists an algebraic subgroup $\AST(A)$ of $\Sp_{2g}/\Q$, called the algebraic Sato--Tate group, such that $G^{1,\Zar}_\ell\simeq \AST(A) \times_\Q \Q_\ell$ for every prime $\ell$.
\end{conjecture}
The Sato--Tate group $\ST(A)$ is then a maximal compact subgroup of $\AST(A)\times_\Q \C$. It should be noted that, following \cite{Ser91}, Banaszak and Kedlaya \cite{BK15a} have given an alternative definition of $\ST(A)$ that also avoids the dependence on $\ell$ and $\iota$. However, this is rendered mostly unnecessary by Theorem~\ref{theorem: MT to AST} below.

The algebraic Sato--Tate group is related to the Mumford--Tate group and the Hodge group. Fix an embedding $k\hookrightarrow \C$. The Mumford--Tate group $\MT(A)$ is the smallest algebraic subgroup $G$ of $\GL(H_1(A_\C,\Q))$ over 
$\Q$ such that $G(\R)$ contains $h(\C^\times)$, where
$$
h\colon \C \rightarrow \End_\R(H_1(A_\C,\R))
$$
is the complex structure on the $2g$-dimensional real vector space $H_1(A_\C,\R)$ obtained by identifying it with the tangent space of $A$ at the identity.
The Hodge group $\Hg(A)$ is the intersection of $\MT(A)$ with $\Sp_{2g}/\Q$. Let $G^{\Zar,0}_\ell$ (resp. $G^{1,\Zar,0}_\ell$) denote the identity component of $G^{\Zar}_\ell$ (resp. $G^{1,\Zar}_\ell$).

\begin{conjecture}[Mumford--Tate conjecture] There is an isomorphism $G^{\Zar,0}_\ell\simeq \MT(A)\times_\Q \Q_\ell$. Equivalently, we have $G^{1,\Zar,0}_\ell\simeq \Hg(A)\times_\Q \Q_\ell$.
\end{conjecture}

The identity component of $\AST(A)$ should thus be the Hodge group $\Hg(A)$. It follows from the definition that $\ST(A)$ has a faithful unitary symplectic representation
$$
\varrho\colon \ST(A)\rightarrow \GL(V)\,,
$$
where $V$ is a $2g$-dimensional $\C$-vector space, which we call the standard representation of $\ST(A)$. Via this representation, we regard $\ST(A)$ as a compact real Lie subgroup of $\USp(2g)$. 

The following result has recently been established by Cantoral Farf\'an--Commelin \cite{CC22}.
\begin{theorem}[Cantoral Farf\'an--Commelin] \label{theorem: MT to AST}
If the Mumford--Tate conjecture holds for $A$, then the algebraic Sato--Tate conjecture also holds for $A$.
\end{theorem}

\subsection{Motivic $L$-functions}\label{subsec:motivic}

As described in \cite[\S8.3.3]{Ser12} to each prime $\p$ of $k$ not dividing $N$ one can attach an element $y_\p$ in the set of conjugacy classes $Y$ of $\ST(A)$ with the property that
$$
\det(1-\varrho_{A,\ell}(\Frob_\p)\NN(\p)^{-1/2}T)=\det(1-\varrho(y_\p) T)\,,
$$ 
where $\Frob_\p$ denotes a Frobenius element at $\p$.
More in general, via Weyl's unitarian trick, any complex representation 
$$
\sigma\colon \ST(A) \rightarrow \GL(V_\chi)\,,
$$ 
say of character $\chi$ and degree $d_\chi$, gives rise to an $\ell$-adic representation
$$
\sigma_{A,\ell}\colon G_k\rightarrow \Aut(V_{\chi,\ell})\,,
$$
where $V_{\chi,\ell}$ is a $\bar{\Q}_\ell$-vector space of dimension $d_\chi$, such that for each prime $\p$ of $k$ not dividing $N$ one has
$$
\det(1-\sigma_{A,\ell}(\Frob_\p)\NN(\p)^{-w_\chi/2}T)=\det(1-\sigma(y_\p) T)\,,
$$
where $w_\chi$ denotes the motivic weight of $\chi$. For a prime $\p$ of $k$, define
$$
L_\p(\chi, T):=\det(1-\sigma_{A,\ell}(\Frob_\p)\NN(\p)^{-w_\chi/2}T\,|\,V_{\chi,\ell}^{I_\p})\,,
$$
where $I_\p$ denotes the inertia subgroup of the decomposition group $G_\p$ at $\p$. The polynomials $L_\p(\chi,T)$ do not depend on $\ell$, and have degree $d_\chi(\p) \leq d_\chi$. Moreover, writing $\alpha_{\p,j}$ for $j=1,\dots,d_\chi(\p)$ to denote the reciprocal roots of $L_\p(\chi,T)$, we have that 
$$
|\alpha_{\p,j}|\leq 1\,.
$$
In fact, for a prime $\p$ not dividing $N$, we have that $d_\chi(\p) = d_\chi$ and $|\alpha_{\p,j}|= 1$.
Therefore, the Euler product 
$$
L(\chi,s):=\prod_{\p}L_\p(\chi,\NN(\p)^{-s})^{-1}
$$
is absolutely convergent for $\Re(s)>1$. We will make strong assumptions on the analytic behavior of the above Euler product. Before, following \cite[\S4.1]{Ser69}, define the positive integer 
$$
B_\chi:=|\disc_{k/\Q}|^{d_\chi}\cdot N_\chi\,,
$$
where $N_\chi$ is the absolute conductor attached to the $\ell$-adic representation $\sigma_{A,\ell}$. For $j=1,\dots,d_\chi$, let $0\leq \kappa_{\chi,j}\leq 1+w_\chi/2$ be the local parameters at infinity (they are semi-integers that can be explicitly computed from the discussion in \cite[\S3]{Ser69}). Define the completed $L$-function
\begin{equation}\label{equation: LandGamma}
\Lambda(\chi,s):=B_\chi^{s/2}L(\chi,s)\Gamma(\chi,s)\,,\qquad  \text{where}\quad\Gamma(\chi,s):=\pi^{d_\chi s/2}\prod_{j=1}^{d_\chi}\Gamma\left(\frac{s+\kappa_{\chi,j}}{2}\right)\,.
\end{equation}
Let $\delta(\chi)$ be the multiplicity of the trivial representation in the character $\chi$ of $\ST(A)$.
\begin{conjecture}[Generalized Riemann hypothesis]\label{conjecture: GRH}
For every irreducible character $\chi$ of $\ST(A)$, the following holds:
\begin{enumerate}[i)]
\item The function $s^{\delta(\chi)}(s-1)^{\delta(\chi)}\Lambda(\chi,s)$ extends to an analytic function on $\C$ of order $1$ which does not vanish at $s=0,1$.
\item There exists $\epsilon \in \C^\times$ with $|\epsilon|=1$ such that for all $s\in \C$ we have
$$
\Lambda(\chi,s)=\epsilon\Lambda(\overline \chi,1-s)\,,
$$
where $\overline \chi$ is the character of the contragredient representation of $\sigma$.
\item The zeros $\rho$ of $\Lambda(\chi,s)$ (equivalently, the zeros $\rho$ of $L(\chi,s)$ with $0<\Re(\rho)<1$) all have $\Re(\rho)=1/2$.
\end{enumerate}
\end{conjecture}

The following estimate of Murty \cite[Prop. 4.1]{Mur85} will be crucial in \S\ref{section: efstc}.
We will need the formulation with the level of generality of \cite[(2.3)]{BK16b}.

\begin{proposition}[Murty's estimate]\label{proposition: Murtyes}
Assume that Conjecture~\ref{conjecture: GRH} holds for the irreducible character $\chi$ of $\ST(A)$. Then 
\begin{equation}\label{equation: Mbound}
\sum_{\NN(\p)\leq x} \chi(y_\p)\log(\NN(\p))= \delta(\chi) x + O(d_\chi \sqrt{x} \log(x)\log(N(x+w_\chi)))\qquad \text{for }x\geq 2\,.
\end{equation}
By applying Abel's summation trick, the above gives
\begin{equation}\label{equation: Mbound2}
\sum_{\NN(\p)\leq x} \chi(y_\p)= \delta(\chi) \Li(x) + O(d_\chi \sqrt{x} \log(N(x+w_\chi)))\qquad\text{for }x\geq 2\,.
\end{equation}
\end{proposition}

\begin{remark}\label{remark: conventionsums}
In \eqref{equation: Mbound} and thereafter, we make the convention that  all sums involving the classes $y_\p$ run over primes $\p$ not dividing $N$. A similar convention applies for sums involving the normalized Frobenius traces $\overline a_\p=\Tr(y_\p)$.
\end{remark}

\begin{remark}
We alert the reader of a small discrepancy between \eqref{equation: Mbound2} and \cite[(2.4)]{BK16b}: in the latter the error term stated is $O(d_\chi\sqrt{x}\log(N(x+d_\chi)))$. We make this precision here, although we note that it has no effect in the subsequent results of \cite{BK16b}. Indeed, in many cases (as those of interest in \cite{BK16b} involving elliptic curves without CM) the weight $w_\chi$ is bounded by the dimension $d_\chi$.
\end{remark}

\begin{remark}
The proof of Proposition \ref{proposition: Murtyes} uses the bound
\begin{equation}\label{equation: condbound}
\log(B_\chi)=O(d_\chi \log(N))\qquad\text{for every character $\chi$ of $\ST(A)$.}
\end{equation}
In order to show \eqref{equation: condbound}, let us recall the definition of $N_\chi$ as a product 
$$
N_\chi:=\prod_\p \NN(\p)^{f_\chi(\p)}
$$
over primes of $k$, where $f_\chi(\p)$ is the exponent conductor at $\p$; this is a nonnegative integer whose definition can be found in \cite[\S2]{Ser69}, for example. If $A$ has good reduction at $\p$, then $f_\chi(\p)$ is zero and so the product is finite. Let $T_{\chi,\ell}$ denote a $\Z_\ell$-lattice in $V_{\chi,\ell}$ stable by the action of $G_\p$. By Grothendieck \cite[\S4]{Gro70}, the exponent conductor can be written as
$$
f_\chi(\p)= \varepsilon_\chi(\p)+\delta_\chi(\p)\,,
$$
where $\varepsilon_\chi(\p)=d_\chi-\dim(V_{\chi,\ell}^{I_\p})$ and $\delta_\chi(\p)$ is the Swan conductor of $V_\chi[\ell]:=T_{\chi,\ell}/\ell T_{\chi,\ell}$ for every $\ell$ coprime to~$\p$. 
Since the kernel of the action  
$$
\overline \sigma_{A,\ell}\colon G_\p \rightarrow \Aut(V_\chi[\ell])
$$ 
on this quotient is contained in the kernel of the action of $G_\p$ on $T_\ell(A)/\ell T_\ell(A)$, 
we have that $\overline \sigma_{A,\ell}$ factors through a finite group $G_{\chi,\p}$ whose order is $O(1)$. Consider the normal filtration of ramification groups
$$
G_{\chi,\p} \supseteq G_0 \supseteq G_1 \supseteq \dots
$$
of $G_{\chi,\p}$. Let us simply write $V$ (resp. $V_i$) for $V_\chi[\ell]$ (resp. $V_\chi[\ell]^{G_i}$). By \cite[Prop. 5.4]{BK94}, we have that 
$$
f_\chi(\p)=\dim(V/V_0) + (a + h(G_1) + 1/(p-1)) e \dim(V/V_1)\,,
$$
where $e$ is the ramification index of $\p$ over $\Q$, $p^{h(G_1)}$ is the exponent of the $p$-group $G_1$ and $p^a$ is the maximal dimension among absolutely simple components of $V/V_1$ as a $G_1$-module. Since $\#G_1$ is $O(1)$, so are $h(G_1)$ and~$a$, because the dimension of an irreducible representation of a group is bounded by the order of the group.
We deduce that
$$
f_\chi(\p)=O(d_\chi)\,,
$$
from which \eqref{equation: condbound} is immediate.
\end{remark}

\section{Effective Sato--Tate Conjecture}\label{section: efstc}

In this section we derive, from the conjectural framework described in \S\ref{section: conjectures}, an effective version of the Sato--Tate conjecture for an arbitrary abelian variety $A$ of dimension $g$ defined over the number field $k$ (see Theorem \ref{theorem: effectivest}). Let~$I$ be a subinterval of $[-2g,2g]$. By \emph{effective} we mean that we provide an upper bound on the error term in the count of primes with normalized Frobenius trace lying in $I$ relative to the prediction made by the Sato--Tate measure. 

The proof is based on the strategy hinted in \cite[\S5]{BK16b}. 
The first step is the construction of a multivariate Vinogradov function  aproximating the characteristic function of the preimage of $I$ by the trace map. This is a continuous periodic function with rapidly decaying Fourier coefficients that generalizes the classical Vinogradov function \cite[Lem. 12]{Vin54}. This construction is accomplished in \S\ref{section: multivinogradov}.

The core of the proof consists in rewriting the Vinogradov function in terms of the irreducible characters of $\ST(A)$ and applying Murty's estimate (Proposition \ref{proposition: Murtyes}) to each of its irreducible constituents. This is the content of \S\ref{section: proof}.

In order to control the size of the coefficients of the character decomposition, we use a result of Gupta \cite[Thm. 3.8]{Gup87} bounding the size and number of nonzero entries of the inverse of the weight multiplicity matrix. Gupta's result and other background material on representations of Lie groups is recalled in \S\ref{section: Lie background}. 

A first analysis does not yield the independence of the implied constant in the $O$-notation from the Lie algebra of $\ST(A)$. This independence is shown to follow from the density of the subgroup generated by the Cartan Hodge circles in the Cartan subgroup. In a result which may be of independent interest (see Theorem~\ref{theorem: Hodge}), this density is shown to follow from the Mumford--Tate conjecture in \S\ref{section: Cartansub}.  

\subsection{Lie group theory background}\label{section: Lie background}

Let $\s$ be a finite dimensional complex semisimple Lie algebra with Cartan subalgebra $\h$ of rank $h$. Let $\Phi\subseteq \h^*$ be a root system for $\s$, $\h_0^*$ be the real vector subspace generated by $\Phi$, and $\mathcal R\subseteq \h^*_0$ denote the lattice of integral weights of $\s$. 

Fix a base $S$ for the root system $\Phi$. The choice of $S$ determines a Weyl chamber in $\h^*_0$ and a partition $\Phi=\Phi^+ \cup \Phi ^-$, where $\Phi^+$ (resp. $\Phi^-$) denotes the set of positive (resp. negative) roots of $\s$. Let $\Cc$ denote the set of dominant weights, that is, the intersection of the set of integral weights $\Rr$ with this Weyl chamber. The choice of a basis of fundamental weights $\{\omega_j\}_j$ determines an isomorphism $\Cc\simeq \Z^h_{\geq 0}$.

For $\lambda, \mu \in \Cc$, the multiplicity $m_\lambda^\mu$ of $\mu$ in $\lambda$ is defined to be the dimension of the space
$$
\Gamma_\lambda^\mu=\{ v\in \Gamma_\lambda\,|\,b(v)=\mu(b)v,\,\,\forall b\in\h\}\,,
$$
where $\Gamma_\lambda$ is the irreducible representation of $\s$ of highest weight $\lambda$.  Write $\rho:=\frac{1}{2}\sum_{\alpha\in\Phi^+}\alpha$ for the Weyl vector and $\W$ for the Weyl group of $\s$. The multiplicity of $\mu$ in $\lambda$ can be computed via Kostant's multiplicity formula 
\begin{equation}\label{equation: Kostant formula}
m_\lambda^\mu=\sum_{w\in\W}\epsilon(w)p(w(\lambda+\rho)-(\mu+\rho))\,,
\end{equation}
where $\epsilon(w)$ is the sign of $w$,
 and $p(v)$ is defined by the identity
$$
\sum_{v\in \Rr} p(v)e^v:=\prod_{\alpha\in \Phi^+}(1-e^\alpha)^{-1}\,,
$$
where we make a formal use of the exponential notation $e^\alpha$ (see \cite[Prop. 25.21]{FH91}).
The natural number $p(v)$ is thus the number of ways to write the weight $v$ as a sum of positive roots with nonnegative coefficients.

Write $\mu\preceq\lambda$ if and only if $\lambda-\mu$ is a sum of positive roots with nonnegative coefficients.
The lattice $\mathcal R\subseteq \h^*_0$ is then partially ordered with respect to the relation $\mu\preceq\lambda$. Relative to this ordering of $\Cc$, the matrix of weight multiplicities $(m^\mu_\lambda)_{\lambda,\mu}$ is lower triangular. Let $(d_\lambda^\mu)_{\lambda,\mu}$ denote the inverse of $(m^\mu_\lambda)_{\lambda,\mu}$.

Gupta has obtained a formula\footnote{In fact, Gupta's result is of a more general nature: it applies to a $q$-analog of $d_\lambda^\mu$. The version of interest to us is obtained by specialization.} in the spirit of Kostant's multiplicity formula for the entries of the inverse matrix $(d_\lambda^\mu)_{\lambda,\mu}$. More precisely, by \cite[Thm. 3.8]{Gup87}, we have that $d_\lambda^\mu=a^\mu_\lambda t_\lambda^{-1}$, where $t_\lambda$ is the size of the stabilizer of $\lambda$ in $\W$ and
$$
a^\lambda_\mu:=\sum_{w\in\W}\epsilon(w)f(w(\lambda+\rho)-\mu)\,.
$$
Here, for each $v\in\mathcal R$, the integer $f(v)$ is defined by
\begin{equation}\label{equation: alammu}
\sum_{v\in\mathcal R}f(v)e^v:=e^\rho\prod_{\alpha\in\Phi^+}(1-e^{-\alpha})\,.
\end{equation}

Let $\varphi$ denote the size of the set of positive roots $\Phi^+$.

\begin{proposition}\label{proposition: Gupta} The sum of the absolute values of the elements in each row (resp column) of $(d_\lambda^\mu)_{\lambda,\mu}$ is bounded by $\#\W\cdot 2^{\varphi}$. In particular $d_\lambda^\mu = O(1)$ and the number of nonzero entries at each row (resp. column) of $(d_\lambda^\mu)_{\lambda,\mu}$ is $O(1)$.
\end{proposition}

\begin{proof}
The proof follows from the aforementioned result by Gupta. Indeed, the sum of the absolute values of the entries at each row (resp. column) of $(d_\lambda^\mu)_{\lambda,\mu}$
is bounded by $\#\W$ times the norm 
$$
\sum_{v\in \Rr} |f(v)|\,.
$$  
But this number is bounded by $2^{\varphi}$, as one observes from (\ref{equation: alammu}). Now the other two statements are implied by the fact that $\varphi$, $\#\W$ can be bounded in terms of $g$, as follows from the general classification of complex semisimple Lie algebras, and thus are $O(1)$. 
\end{proof}

For $\lambda \in \Cc$, write $\lambda$ as a nonnegative integral linear combination $\sum_{j=1}^s m_j \omega_j$  of the fundamental weights  and define
$$
||\lambda||_{\fund}:=\max_j\{ m_j\}\,.
$$

\begin{proposition}\label{proposition: dimbound}
The previous definition has the following properties.
\begin{enumerate}[i)]
\item $\dim(\Gamma_\lambda)=O( ||\lambda||_{\fund}^{\varphi})$ for every $\lambda\in \Cc$.\vspace{0.1cm}
\item $\dim (\Gamma_{\lambda'})= O(\dim (\Gamma_\lambda))$ for every $\lambda,\lambda'\in \Cc$ with $\lambda'\preceq \lambda$.\vspace{0.1cm}
\item For every $\lambda\in \Cc$, the motivic weight of the $\ell$-adic representation $(\Gamma_\lambda)_{A,\ell}$ attached to $\Gamma_\lambda$ as in \S\ref{section: conjectures} is $O(||\lambda||_{\fund})$.
\end{enumerate}

\end{proposition}

\begin{proof}
For i), recall Weyl's dimension formula \cite[Cor. 1 to Thm. 4, Chap. VII]{Ser87}, which states
$$
\dim(\Gamma_\lambda)=\prod_{\alpha \in \Phi^+}\frac{(\lambda+\rho,\alpha)}{(\rho,\alpha)}\,,
$$
where $(\cdot, \cdot)$ denotes a $\W$-invariant positive definite form on the real vector space $\h_0^*$ spanned by the base~$S$.
This trivially implies
$$
\dim(\Gamma_\lambda) \asymp \prod_{\alpha\in \Phi^+}(\lambda,\alpha)\,.
$$
It remains to show that $(\lambda,\alpha)=O( ||\lambda||_{\fund})$ for every $\alpha\in \Phi^+$. Let $\alpha_j$, for $j=1,\dots,h$, be the constituents of the base $S$, the so-called simple roots. The desired result follows from the following relation linking simple roots and fundamental weights 
\begin{equation}\label{equation: changebase}
2\frac{(\omega_l,\alpha_j)}{(\alpha_j,\alpha_j)}=\delta_{lj}\,.
\end{equation}
As for ii), suppose that the expression of $\lambda \in \Cc$ (resp. $\lambda '\in \Cc$) as a nonnegative linear combination of the simple roots is $\sum_{j=1}^h r_j \alpha_j$ (resp. $\sum_{j=1}^h r_j' \alpha_j$). Note that $\lambda'\preceq \lambda$ implies that $r_j'\leq r_j$. Therefore
$$
\dim(\Gamma_{\lambda'})=O\left(\prod_{\alpha\in \Phi^+}(\lambda',\alpha)\right)=O\left(\prod_{\alpha\in \Phi^+}(\lambda,\alpha)\right)=O(\dim(\Gamma_\lambda))\,.
$$
Part iii) is a consequence of the weight decomposition of $\Gamma_\lambda$. 
\end{proof}

\subsection{A multivariate Vinogradov function}\label{section: multivinogradov}

The main result of this section is Proposition~\ref{proposition: vinogradov}, which is a generalization of \cite[Lemma 12]{Vin54}.
Let $q\geq 1$ be a positive integer. We will write $\tb$ to denote the $q$-tuple $(\theta_1,\dots, \theta_q)\in \R^q$ (a similar convention applies to $\zb$, $\db$, etc). We also write $\mb$ to denote $(m_1,\dots,m_q)\in \Z^q$.
We will say that a function $h\colon \R^q\rightarrow \R$ is periodic of period $1$ if it is so in each variable.

For $\db = (\delta_1, \dots, \delta_q)\in [0,1)^q$, denote by $R(\db)$ the parallelepiped $\prod_{j=1}^q[-\delta_j,\delta_j]$. Set also the multiplier
$$
\nu(\mb,\db):=\prod_{j=1}^{q}\nu(m_j, \delta_j)\,, 
\qquad \text{where}\quad 
\nu(m_j, \delta_j) :=\begin{cases} 1 & m_j = 0\,,\\ 
&\\
\displaystyle\frac{\sin(2 \pi m_j \delta_j)}{2 \pi m_j \delta_j} & m_j \neq 0 . \end{cases}
$$

\begin{lemma}\label{lemma: Fourieraverage}
Suppose that $h\colon \R^q \rightarrow \R$ admits a Fourier series expansion as 

$$
h(\tb) = \sum_{\mb \in \Z^q} c_{\mb}(h) e^{2 \pi i (\mb\cdot\tb)}\,,\qquad 
\text{where}\quad 
c_{\mb}(h) := \int_{[0,1]^q}h(\tb) e^{- 2 \pi i (\mb\cdot\tb)} d\tb.
$$
For $\db \in [0,1)^q$, define 
\begin{equation}\label{equation: average}
f(\tb) := \left(\prod_{j=1}^q\frac{1}{2\delta_j}\right) \int_{R(\db)} h(\tb + \zb) d\zb.
\end{equation} 
Then we have that
\begin{equation}\label{eq:fourierrec} 
c_{\mb}(f) = c_{\mb}(h)\nu(\mb, \db)\,.
\end{equation}
\end{lemma}

\begin{proof} The proof follows the same lines as Vinogradov's one-dimensional version.  
We have 

\begin{align*}
c_{\mb}(f) & = \int_{[0,1]^q}f(\tb) e^{- 2 \pi i (\mb\cdot \tb) } d\tb \\
& =\left(\prod_{j=1}^q\frac{1}{2\delta_j}\right)   \int_{[0,1]^q} \int_{R(\db)} h(\tb+ \zb) e^{- 2 \pi i (\mb\cdot \tb) }d\zb d\tb\\
& = \left(\prod_{j=1}^q\frac{1}{2\delta_j}\right) \int_{R(\db)}\int_{[0,1]^q}  h(\tb+\zb) e^{- 2 \pi i (\mb\cdot\tb)}d\tb d\zb\,. \\ 
\end{align*}
Setting $ \ttb = \tb+\zb$ so that $\tb = \ttb - \zb$ and $d\tb = d\ttb$ in the above equation, we obtain 
\begin{align*}c_{\mb} (f) & = \int_{[0,1]^q}   h(\ttb) e^{- 2 \pi i (\mb \cdot \ttb)} d\ttb \cdot \prod_{j=1}^q\frac{1}{2\delta_j}\int_{-\delta_j}^{\delta_j}    e^{2\pi i m_j z_j} dz_j     \\
&= c_{\mb}(h) \prod_{j=1}^q\frac{1}{2\delta_j}\int_{-\delta_j}^{\delta_j}    e^{2\pi i m_j z_j} dz_j\,.
\end{align*} 
For $m_j = 0$ the corresponding term in the product is 
$$
\frac{1}{2\delta_j}\ \int_{-\delta_j}^{\delta_j}  1\,dz_j = 1 = \nu(0, \delta_j)\,.
$$ 
For $m_j \neq 0$ the corresponding term becomes  
$$ 
\frac{1}{2\delta_j}\int_{-\delta_j}^{\delta_j}  e^{2\pi i m_j z_j} dz_j =\frac{1}{2\delta_j} \cdot  \frac{e^{2\pi i m_j \delta_j}  - e^{-2\pi i m_j \delta_j} }{2 \pi i m_j } = \frac{\sin(2 \pi m_j \delta_j)}{2 \pi m_j \delta_j} = \nu(m_j, \delta_j).
$$
\noindent The desired formula follows. 
\end{proof}

For $1\leq j\leq q$, let $\pi_j\colon [0,1]^q\rightarrow [0,1]^{q-1}
$ be the map that sends $\tb\in [0,1]^q$ to the $(q-1)$-tuple obtained from $\tb$ by suppressing its $j$-th component. 
For $\vartb\in[0,1]^{q-1}$, define $X_j(\vartb)=\pi_j^{-1}(\vartb)$.
\begin{proposition}\label{proposition: vinogradov}
Let $T\colon \R^q\rightarrow \R$ be a differentiable function satisfying the following hypotheses:
\begin{enumerate}
\item It is periodic of period $1$. 
\item There exists a real number $K>0$ such that $|\nabla T(\tb)|\leq K$ for every $\tb\in \R^q$.
\item There exists a positive integer $C>0$ such that, for every $\gamma\in \R$, $1\leq j\leq q$, and $\vartb\in [0,1]^{q-1}$, we have
$$
\#(T^{-1}(\gamma)\cap X_j(\vartb))\leq C\,.
$$
\end{enumerate}
Let $\alpha$, $\beta$, $\Delta$ be real numbers satisfying
\begin{equation}\label{equation: Deltaconstraint}
\Delta > 0 \,,\qquad 2\Delta \leq \beta-\alpha \,.
\end{equation}
Let $I$ denote the open interval $(\alpha,\beta)$. By \eqref{equation: Deltaconstraint} we can define the disjoint sets 
$$
\begin{array}{l}
R_1:=R_1(\Delta,\alpha,\beta):=T^{-1}((\alpha+\Delta,\beta-\Delta))\cap [0,1]^q\,,\\[6pt]
R_0:=R_0(\Delta,\alpha,\beta):=T^{-1}(\R\setminus [\alpha-\Delta,\beta+\Delta])\cap [0,1]^q.
\end{array}
$$
Then for every positive integer $r\geq 1$, there exists a continuous function $D:=D_{\Delta,I}\colon \R^q\rightarrow \R$ periodic of period $1$ satisfying the following properties:
\begin{enumerate}[i)]
\item For $\tb \in R_1$, we have $D(\tb)=1$. 
\item For $\tb \in R_0$, we have $D(\tb)=0$.
\item $D(\tb)$ has a Fourier series expansion of the form
$$
D(\tb)=\sum_{\mb\in \Z^q}c_{\mb}e^{2 \pi i (\mb\cdot \tb)}\,,
$$
where $c_{\boldsymbol 0}=\int_{T^{-1}((\alpha,\beta))\cap[0,1]^q}d\tb$ and for all $\mb\not =0$ we have
$$
|c_{\mb}|\leq \min\left\{ |c_{\oub}|, \left\{\frac{C}{\pi\max_j\{|m_j|\}}\prod_{j=1, m_j\not=0}^q\min\left\{1, \left(\frac{ rK\sqrt q}{2\pi|m_j|\Delta}\right)^\rho\right\} \right\}_{\rho=0,\dots,r}\right\}\,.
$$
\end{enumerate}
\end{proposition}

\begin{proof}
Start by defining the function $\psi_0$ periodic of period 1 as
$$
\psi_0(\tb) := 
\begin{cases} 
1 & \text{ if } \tb \in T^{-1}((\alpha,\beta)),\\
0 & \text{ if } \tb \in T^{-1}(\R\setminus [\alpha,\beta]), \\
1/2 & \text{ if } \tb \in T^{-1}(\alpha) \cup T^ {-1}(\beta).
 \end{cases}
$$
Then we clearly have
$$
c_{\oub}(\psi_0) = \int_{T^{-1}((\alpha,\beta))\cap [0,1]^q}d\tb\,,
$$
and for $\mb \neq 0$ we find the bound
\begin{equation}\label{equation: coef0}
 |c_{\mb}(\psi_0)| = \left|\int_{T^{-1}((\alpha,\beta))\cap [0,1]^q} e^{-2 \pi i (\mb \cdot \tb)} d\tb\right| \leq |c_{\oub}(\psi_0)|.
\end{equation}
We next derive an alternative upper bound for $c_{\mb}(\psi_0)$. Let $m$ denote $\max_l\{|m_l|\}$ and let $j$ be such that $m=|m_j|$. Then by Fubini's theorem we have
$$
c_{\mb}(\psi_0) = \int_{T^{-1}((\alpha,\beta))\cap [0,1]^{q-1}}\left(\int_{T^{-1}((\alpha,\beta))\cap X_{j}(\pi_j(\tb))}e^{-2 \pi i m \theta_{j}} d\theta_{j}\right)e^{-2 \pi i \pi_j(\mb) \cdot \pi_j(\tb)} d\pi_j(\tb)\,.
$$
By condition (3) we have that $T^{-1}((\alpha,\beta))\cap X_{j}(\pi_j(\tb))$ is a union of at most $C$ intervals. It follows that
\begin{equation}\label{equation: coef0second}
|c_{\mb}(\psi_0)|\leq 2C\frac{1}{2\pi m}=\frac{C}{\pi m}\,.
\end{equation}
Fix $\delta >0$ such that $r\sqrt{q}K\delta=\Delta$ and set $\db:=(\delta,\dots,\delta)$. By averaging over the region $R(\db)$ as in \eqref{equation: average}, we recursively define the function $ \psi_\rho$, for $1 \leq \rho \leq r$, as 

\begin{equation}\label{eq:recpsi}
  \psi_\rho (\tb) = \frac{1}{(2\delta)^q} \int_{R(\db)} \psi_{\rho -1} (\tb+\zb) d\zb\,.
\end{equation}

We will prove inductively that: 

\begin{enumerate}[a)]
\item $ \psi_\rho (\tb) \in \R$ for all $\tb$.
\item $ 0 \leq  \psi_\rho (\tb) \leq 1$ for all $\tb$.
\item $  \psi_\rho (\tb)  = 1 $ for $\tb\in T^{-1}((\alpha+\rho \Delta/r, \beta-\rho\Delta/r))$.
\item $  \psi_\rho (\tb)  = 0 $ for $\tb\in T^{-1}(\R\setminus [\alpha-\rho \Delta/r, \beta+\rho\Delta/r])$.

\item $c_{\oub} (\psi_\rho) = c_{\oub} (\psi_0)$.
\item For $\mb \neq 0$, there is an equality 
$$
c_{\mb}(\psi_\rho) = c_{\mb}(\psi_0) \nu(\mb,\db)^\rho. 
$$
\end{enumerate}

The initial function $\psi_0$ satisfies all these properties. Now assume that $\psi_{\rho-1} $ also satisfies them. Then it is clear that $\psi_\rho$ will satisfy the first two. In order to prove c), note that for $\zb\in R(\db)$, the multivariate mean value theorem gives
\begin{equation}\label{equation: multvalthm}
|T(\tb+\zb)-T(\tb)|\leq K|z|\leq K\sqrt q \delta=\frac{\Delta}{r}\,.
\end{equation}
Let $\tb\in T^{-1}((\alpha+\rho \Delta/r, \beta-\rho\Delta/r))$. By \eqref{equation: multvalthm}, we have that $\tb+\zb\in T^{-1}((\alpha+(\rho-1) \Delta/r, \beta-(\rho-1)\Delta/r))$ and therefore
$$
\psi_\rho(\tb)=\frac{1}{(2\delta)^q}\int_{R(\db)}\psi_{\rho-1}(\tb+\zb)d\zb=\frac{1}{(2\delta)^q}\int_{R(\db)}d\zb=1\,,
$$ 
where in the middle equality we have used the induction hypothesis. The proof of d) is analogous.
Properties e) and f) are immediate from Lemma~\ref{lemma: Fourieraverage}. 

Note that f), \eqref{equation: coef0}, and \eqref{equation: coef0second} imply that 
$$
|c_{\mb}(\psi_\rho)| \leq |c_{\mb}(\psi_0)| \leq  \min\left\{|c_{\oub}(\psi_0)|,\frac{C}{\pi m}\right\}\,.
$$
To conclude, take $D := \psi_r$, and the proposition follows from f) and the fact that, for $m_j\not = 0$, we have
$$
|\nu(m_j,\delta)|\leq \min\left\{1,\frac{r K \sqrt q}{2\pi |m_j|\Delta }\right\}\,.
$$
\end{proof}

\subsection{The Cartan subgroup}\label{section: Cartansub} As in the previous sections, $A$ denotes an abelian variety of dimension $g$ defined over the number field $k$. From now on we will assume moreover that its Sato--Tate group $\ST(A)$ is connected. 

Since $\ST(A)$ is reductive, its complexified Lie algebra $\g$ is the product of a semisimple Lie algebra $\s$ and an abelian Lie algebra $\fa$.
Recall the notations from \S\ref{section: Lie background} relative to $\s$; in particular, $\h$ is a Cartan subalgebra for $\s$ and $h$ denotes the rank of $\h$. Given $(\theta_1,\dots,\theta_g)\in\R^g$, set
$$
d(\theta_1,\dots,\theta_g):=\diag(e^{2\pi i \theta_1},\dots,e^{2\pi i \theta_g}, e^{-2\pi i \theta_1},\dots,e^{-2\pi i\theta_g})\,.
$$
Let $a$ denote the rank of $\fa$ and let $q=h+a$ be the rank of $\g$. As in \S\ref{section: multivinogradov}, write $\tb$ to denote $(\theta_1, \dots, \theta_q)\in \R^q$.
We may choose $\ab_{q+1},\dots,\ab_g \in \Z^{g-q}$ such that the image $H$ of the map
\begin{equation}\label{equation: mapiota}
\iota \colon \R^q \rightarrow \ST(A) \,,
\qquad
\iota(\tb)= d(\theta_1,\dots, \theta_q, \tb\cdot \ab_{q+1}, \dots, \tb\cdot \ab_{g})
\end{equation} 
has complexified Lie algebra isomorphic to $\h\times \fa$. We then say that $H$ is a Cartan subgroup of $\ST(A)$. For notational purposes, it will be convenient to let $\ab_1,\dots, \ab_q$ denote the standard basis of $\Z^q$. Let $a_{l,j}$ denote the $j$-th component of $\ab_l$. 

Consider the map
\begin{equation}\label{equation: mapT}
T\colon \R^q \stackrel{\iota}{\rightarrow}H\subseteq \ST(A) \xrightarrow{\Tr} [-2g, 2g]\,,\qquad T(\tb)=\sum_{j=1}^g 2\cos(2\pi\ab_j\cdot \tb)\,.
\end{equation}
In the next section, we will apply the construction of a Vinogradov function attached to the map $T$, as seen in \S\ref{section: multivinogradov}. In order to control $|\nabla(T)|$ we need to control the size of $\ab_{q+1}, \dots, \ab_g$. The following form of the Mumford--Tate conjecture serves such a purpose.

By a \emph{Cartan Hodge circle} we will mean the image of any homomorphism
$$
\varphi\colon \R\rightarrow H
$$
such that $\varphi(\theta)$ has $g$ eigenvalues equal to $e^{2\pi i \theta}$ and $g$ eigenvalues equal to $e^{-2\pi i \theta}$. The following statement is a refinement of the ``Hodge condition'' included among the ``Sato--Tate axioms'' stated in \cite[Proposition~3.2]{FKRS12}, \cite[Remark~2.3]{FKS16} (see also \cite[8.2.3.6(i)]{Ser12}).

\begin{theorem}\label{theorem: Hodge}
Suppose that the Mumford--Tate conjecture holds for $A$. Then the group $H$ is generated by Cartan Hodge circles.
\end{theorem}
\begin{proof}
In case $\ST(A)$ is abelian, then it is equal to $H$ and the claim is that $\ST(A)$ itself is generated by Hodge circles. This follows from \cite[Proposition~3.2]{FKRS12} as augmented in \cite[Remark~2.3]{FKS16}.

We next reduce the general case to the previous paragraph, by arguing as in the proof of Deligne's theorem on absolute Hodge cycles. Recall that the Mumford--Tate group of $A$
is the smallest $\Q$-algebraic subgroup of $\GL(H_1(A_{\C}^{\mathrm{top}}, \Q))$ whose base extension to $\R$ contains the action of the Deligne torus $\mathrm{Res}_{\C/\R}(\mathbb{G}_m)$ coming from the Hodge structure.
Under our hypotheses on $A$, we may recover $\ST(A)$ by taking the Mumford--Tate group, taking the kernel of the determinant to get the Hodge group, then taking a maximal compact subgroup.

By the proof of \cite[Proposition~6.1]{Del82}, there exists an algebraic family of abelian varieties containing $A$ as a fiber such that on one hand, the generic Mumford--Tate group is equal to that of $A$, and on the other hand there is a fiber $B$ whose Mumford--Tate group is a maximal torus in $A$. Using the previous paragraph, we see that the desired assertion for $A$ follows from the corresponding assertion for $B$, which we deduce from the first paragraph.
\end{proof}

\begin{lemma}\label{lemma: Cartanmult}
Suppose that the Mumford--Tate conjecture holds for $A$. Then $|a_{l,j}|=O(1)$.  
\end{lemma}

\begin{proof}
Write $\Aa$ for the matrix $(a_{l,j})_{l,j}$. Giving a Cartan Hodge circle amounts to giving a vector $\vb\in \{\pm 1\}^q$ such that
\begin{equation}\label{equation: CHcircle}
\Aa\vb^t=\ub^t
\end{equation}
where $\ub\in \{\pm 1\}^g$ has $g$ entries equal to $1$ and $g$ entries equal to $-1$ (and $\vb^t$,~$\ub^t$ denote the transposes of $\vb$, $\ub$). By Theorem~\ref{theorem: Hodge}, there exist $q$ linearly independent vectors $\vb$ satisfying an equation of the type \eqref{equation: CHcircle}. Let $\vb_j$, for $j=1,\dots,q$, denote these vectors, and let $\ub_j\in\{\pm 1\}^g$ denote the corresponding constant terms in the equation that they satisfy. Let $v_{j,l}$ (resp. $u_{j,l}$) denote the $l$-th component of $\vb_j$ (resp. $\ub_j$). Write $\Vv$ (resp. $\Uu$) for the matrix $(v_{l,j})_{j,l}$ (resp. $(u_{l,j})_{j,l}$). Since $\Vv$ is invertible, we have
$$
\Aa=\mathcal \Vv^{-1}\Uu\,.
$$
The lemma now follows immediately from the fact that all the entries of $\Vv$ and $\Uu$ are $\pm 1$.
\end{proof}

\begin{lemma}\label{lemma: gradcontrol}
Suppose that the Mumford--Tate conjecture holds for~$A$. Then the map $T\colon \R^q\rightarrow [-2g,2g]$ from \eqref{equation: mapT} satisfies conditions (1), (2), and (3) of Proposition \ref{proposition: vinogradov}. Moreover, both constants $K$ and $C$ appearing respectively in (2) and (3) are $O(1)$.
\end{lemma}

\begin{proof}
An easy computation shows that for every $\tb\in\R^ q$ we have that
$$
\nabla(T)(\tb)=-4\pi\left(\sum_{j=1}^g\sin(2\pi \ab_j\cdot \tb)a_{j,1},\dots,\sum_{j=1}^g\sin(2\pi \ab_j\cdot \tb)a_{j,g}\right)\,,
$$
from which the desired bound $|\nabla(T)(\tb)|=O(1)$ is a consequence of Lemma~\ref{lemma: Cartanmult}.

As for (3), let $1\leq j\leq q$, and fix $\pi_j(\tb)\in \R^{q-1}$ and $\gamma\in\R$. Suppose that $\vartheta\in[0,1]$ satisfies
$$
T(\theta_1,\dots,\theta_{j-1},\vartheta,\theta_{j+1},\dots,\theta_q)=\gamma\,.
$$
This means that there exist real numbers $r_l$ depending exclusively on $\pi_j(\tb)$ such that
$$
\sum_{l=1}^g2\cos(2\pi a_{lj}\vartheta +r_l)=\gamma\,.
$$
Let $N=\max_{l}\{a_{lj}\}$. By the identity $\cos(2\pi a_{lj}\vartheta +r_l)=\cos(2\pi a_{lj}\vartheta)\cos(r_l)-\sin(2\pi a_{lj}\vartheta)\sin(r_l)$ and de Moivre's formula, we deduce that there exist polynomials $p,q\in\R[x]$ of degree $\leq N$ such that
$$
p(\cos(2\pi\vartheta))-q(\sin(2\pi\vartheta))=\gamma\,.
$$
If we write $q(x)=\sum_nb_nx^n$, the above equality implies that $\cos(2\pi\vartheta)$ is a root of
$$
r(x)=(\gamma-p(x)+\sum_{n}b_{2n}(1-x^2)^{n})^2-(1-x^2)\big(\sum_n b_{2n+1}(1-x^2)^{n}\big)^ 2\,.
$$
Since $r(x)$ has degree $\leq 2N$, we find that $\cos(2\pi\vartheta)$ is limited to $2N$ values. This implies that $\vartheta$ is limited to~$4N$ values, and we conclude by applying Lemma~\ref{lemma: Cartanmult}, which shows that $N=O(1)$.
\end{proof}

\subsection{Main theorem}\label{section: proof}

In this section we prove an effective version of the Sato--Tate conjecture building on the results obtained in all of the previous sections. 

Let $\mu$ be the pushforward of the Haar measure of $\ST(A)$ on $[-2g,2g]$ via the trace map. We refer to \cite[\S8.1.3, \S8.4.3]{Ser12} for properties and the structure of this measure. It admits a decomposition $\mu=\mu^{\mathrm{disc}}+\mu^{\mathrm{cont}}$, where $\mu^{\mathrm{disc}}$ is a finite sum of Dirac measures and $\mu^{\mathrm{cont}}$ is a measure having a continuous, integrable, and even $\Cc^\infty$ density function with respect to the Lebesgue measure outside a finite number of points. Since we will assume that $\ST(A)$ is connected, we will in fact have that $\mu^\mathrm{disc}$ is trivial (see \cite[\S8.4.3.3]{Ser12}). 

Attached to the Lie algebra $\g$ of $\ST(A)$, let $\varepsilon:=\varepsilon_\g$ be as defined in \eqref{equation: nuinterval} and $\nu:=\nu_\g\colon \R_{>0}\rightarrow \R_{>0}$ be as defined in \eqref{equation: epsilondefinition}. For an interval $I\subseteq [-2g,2g]$, recall that we denote by $\delta_I$ the characteristic function of $I$.

\begin{theorem}\label{theorem: effectivest}
Let $k$ be a number field and $g$ a positive integer.
Let~$A$ be an abelian variety defined over~$k$ of dimension $g$, absolute conductor $N$, and such that $\ST(A)$ is connected.  Suppose that the Mumford--Tate conjecture holds for $A$ and that Conjecture~\ref{conjecture: GRH} holds for every irreducible character $\chi$ of $\ST(A)$. 
For each prime~$\p$ not dividing $N$, let $\overline a_{\p}$ denote the normalized Frobenius trace of $A$ at $\p$.
Then for all nonempty subintervals $I$ of $[-2g,2g]$, we have
$$
\sum_{\NN(\p)\leq x}\delta_{I}(\overline a_{\p})=\mu(I)\Li(x)+O\left(\frac{x^{1-\varepsilon_\g}\log(Nx)^{2\varepsilon_\g}}{\log (x)^{1-4\varepsilon_\g}}\right)\quad \text{for }x\geq x_0\,,
$$
where $x_0=O\left(\nu_\g(|I|)\log(2N)^2\log(\log(4N))^4\right)$.
\end{theorem}

Let us resume the notations of \S\ref{section: Lie background} relative to the semisimple algebra $\s$. Thus, $\h$ is a Cartan subalgebra for $\s$ of rank $h$, $\mathcal R\subseteq \h_0^*$ is the lattice of integral weights, $\W$ is the Weyl group of $\s$, $\Cc$ denotes the integral weights in a Weyl chamber, and $\omega_1,\dots, \omega_h$ are the fundamental weights. Let $a$ denote the rank of $\fa$, so that $q=h+a$. Before starting the proof we introduce some additional notations. 

Recall the map $\iota\colon \R^q\rightarrow \ST(A)$ from \eqref{equation: mapiota}. Without loss of generality, we may assume that the decomposition $\R^q=\R^h\times \R^a$ is such that the complexification of the Lie algebra of $\iota(\R^h)$ (resp. $\iota(\R^a)$) is $\h$ (resp. $\fa$).
Let us write $\tb_h$ (resp. $\tb_a$) for the projection of $\tb$ onto $\R^h$ (resp. $\R^a$).

From now on we fix a $\Z$-basis $\psi_1,\dots,\psi_q$ of the character group $\hat H$ of $H$: for $1\leq j\leq h$, the character $\psi_j$ is induced by the fundamental weight $\omega_j$ of $\s$; for $h+1 \leq j\leq q$, we set
$$
\psi_j(\iota(\tb))= e^{2\pi i \theta_j}\,.
$$  

The action of $\W$ on $\h_0^*$ induces an action of $\W$ on the character group $\hat H$ of $H$. 
We may define an action of $\W$ on $[0,1]^q$ by transport of structure: given $w\in \W$, let $w(\tb)$ be defined by
\begin{equation}\label{equation: unipotbasis}
\psi_j(\iota(w(\tb)))=w(\psi_j)(\iota(\tb))\qquad \text{for all }j=1,\dots, q\,.
\end{equation}
Of course the action of $\W$ restricts to the first factor of the decomposition $[0,1]^q=[0,1]^h\times [0,1]^a$. Note that the map $\iota$ from \eqref{equation: mapiota} induces an isomorphism 
$$
\iota \colon [0,1]^q/\W\xrightarrow{\sim } \Conj(\ST(A))\,.
$$
Recall the elements $y_\p \in \Conj(\ST(A))$ introduced in \S\ref{section: conjectures}. Let $\tb_\p \in [0,1]^q/\W$ be the preimage of $y_\p$ by the above isomorphism. 

Consider the map $T\colon \R^q\rightarrow [-2g,2g]$ defined in \eqref{equation: mapT}. Note that $T(\tb_\p)$ is well defined since $T$ factors through $[0,1]^q/\W$, and it is equal to the normalized Frobenius trace $\overline a_\p$. Let $K$ and $C$ denote the constants of Lemma \ref{lemma: gradcontrol} relative to the map $T$. 

Let the interior of $I$ be of the form $(\alpha,\beta)$ for $-2g\leq \alpha < \beta \leq 2g$. Let $\Delta>0$ be any real number satisfying the constraint~\eqref{equation: Deltaconstraint} relative to $\alpha$, and $\beta$ (arbitrary for the moment and to be specified in the course of the proof of Theorem \ref{theorem: effectivest}).

Let $D:=D_{\Delta,I}\colon \R^q\rightarrow \R$ be the Vinogradov function produced by Proposition~\ref{proposition: vinogradov}, when applied to $\alpha$, $\beta$, $\Delta$, and $T$, and relative to the choice of a positive integer $r\geq 1$ (arbitrary for the moment and to be specified in the course of the proof of Theorem \ref{theorem: effectivest}). Define
\begin{equation}\label{equation: defofF}
F:=F_{\Delta,I}\colon \R^q \rightarrow \R\,,\qquad F(\tb):=\frac{1}{\#\W}\sum_{w\in \W} D\left( w(\tb) \right)\,.
\end{equation}
Notice that $F(\tb_\p)$ is well-defined since $F$ has been defined as an average over $\W$. In consonance with Remark~\ref{remark: conventionsums}, we make the convention that sums involving the elements $\tb_\p$ run over primes $\p$ not dividing $N$.

\begin{lemma}\label{lemma: dF} If part i) of Conjecture \ref{conjecture: GRH} holds for every irreducible character $\chi$ of $\ST(A)$, then
$$
\sum_{\NN(\p)\leq x}\delta_{I}(\overline a_{\p})=
\sum_{\NN(\p)\leq x} F_{\Delta,I}(\tb_\p) + O\left(\Delta\Li(x)\right)\qquad\text{for every $\Delta$ satisfying \eqref{equation: Deltaconstraint} and every $x\geq 2$}\,.
$$
\end{lemma}

\begin{proof}
Let $Y_\alpha$, $Y_\beta$ denote the preimages of $\alpha$, $\beta$ by the map $T$ in $[0,1]^q$. Let $\Ss=\{\ssb\in [0,1]^q\,|\,\nabla(T)(\ssb)=0\}$ denote the set of critical points of $T$. Let $\Rr$ be the set 
$$
\{\tb \in [0,1]^q\,|\,\pi_j(\tb)=\pi_j(\ssb)\text{ for some }1\leq j \leq q, \ssb\in\Ss\}\,.
$$ 
Since $T$ satisfies property (3) of Proposition~\ref{proposition: vinogradov}), by Lemma \ref{lemma: gradcontrol} the intersections of $Y_\alpha$, $Y_\beta$ with $\Rr$ are finite (and, in fact, even of cardinality $O(1)$). Let $W_\alpha,W_\beta$ denote the intersections of $Y_\alpha,Y_\beta$ with the complement of $\Rr$. 
 
We claim that $W_\alpha,W_\beta$ have volume $O(1)$ as $(q-1)$-dimensional Riemannian submanifolds of $[0,1]^q$. 
Before showing the claim, we note that it implies the lemma. Indeed, as functions over $[0,1]^q$, the characteristic function of $T^{-1}(I)$ and $F_{\Delta,I}$ only differ (by construction of the latter) over the $\W$-translates of tubular neighborhoods $B(Y_\alpha,r_\Delta)$ and $B(Y_\beta,r_\Delta)$ of $Y_\alpha$ and $Y_\beta$ of radii $r_\Delta=O(\Delta)$. If $W_\alpha,W_\beta$ have volume $O(1)$, then $B(Y_\alpha,r_\Delta),B(Y_\beta,r_\Delta)$ have volume $O(\Delta)$. Weyl's integration formula \cite[Chap. IX, \S 6, Cor. 2, p. 338]{Bou03} together with the fact that the absolute value of Weyl's density function is $O(1)$ (see \cite[Chap. IX, \S6, p. 335]{Bou03}) imply that the Haar measure of the $\W$-translates of $B(Y_\alpha,r_\Delta)$ and $B(Y_\beta,r_\Delta)$ is $O(\Delta)$. Then the lemma follows from the equidistribution of $\theta_\p$ implied by part i) of Conjecture \ref{conjecture: GRH} and the prime number theorem.

We now turn to show that $W_\alpha$ has volume $O(1)$ (and the same argument applies to $W_\beta$). For $1\leq j\leq q$, define
$$
\Vv_j=\left\{\tb \in [0,1]^q\, | \, \frac{\partial T}{\partial \theta_j}(\tb)\geq \frac{\partial T}{\partial \theta_l}(\tb) \text{ for every } 1\leq l\leq q\right\}\,.
$$
It suffices to show that $W_\alpha \cap \Vv_j$ has volume $O(1)$ for every $j$. By symmetry, we may assume that $j=q$, which will be convenient for notational purposes. Let $\Zz_{\alpha,q}$ denote the interior of the image of $W_\alpha\cap \Vv_q $ by the projection map $\pi_q\colon [0,1]^q\rightarrow [0,1]^{q-1}$. For $\vartb\in\Zz_{\alpha,q}$, choose $\tilde \vartb\in W_\alpha \cap \Vv_q$ such that $\pi_q(\tilde\vartb)=\vartb$. By the implicit function theorem there exist a neighborhood $\Uu_{\vartb}\subseteq \Zz_{\alpha,q}$ of $\vartb$ and a differentiable function $g_{\vartb}\colon \Uu_{\vartb}\rightarrow \R$ such that 
$$
\tilde \vartb=(\vartb,g(\vartb))\,\quad\text{and}\quad (\ttb,g(\ttb))\in W_\alpha\cap\Vv_q \quad\text{for every }\ttb\in \Uu_{\alpha}.
$$  
The lifts $\tilde\vartb$ can be compatibly chosen so that the functions $g_{\vartb}$ glue together into a differentiable function $g\colon \Zz_{\alpha,q}\rightarrow W_\alpha \cap \Vv_q$. Then Lemma \ref{lemma: gradcontrol} provides the following bound for the volume of $W_\alpha \cap \Vv_q$ 
$$
O\left(\int_{\Zz_{\alpha,q}}\prod_{j=1}^{q-1}\left(1+\left(\frac{\partial g}{\partial \theta_j}(\vartb)\right)^2\right)^{1/2}d\vartb\right)
=O\left(\int_{\Zz_{\alpha,q}}\prod_{j=1}^{q-1}\left(1+\left(\frac{\partial T}{\partial \theta_q}\right)^{-2}\left(\frac{\partial T}{\partial \theta_j}\right)^2(\vartb,g(\vartb))\right)^{1/2}d\vartb\right)=O(1)\,,
$$
which completes the proof.
\end{proof}

\begin{proof}[Proof of Theorem \ref{theorem: effectivest}]
The choice of a basis of fundamental weights $\omega_1,\dots,\omega_h$ gives an isomorphism 
$$
\Z^h\simeq \mathcal R=\W\cdot \Cc\,,
$$
by means of which, from now on, we will view integral weights of $\s$ as elements in $\Z^h$. Similarly, the choice of the basis elements of \eqref{equation: unipotbasis} provides an isomorphism between the lattice of integral weights of $\fa$ and $\Z^a$. For a weight $\mb\in \Z^q$, let $\mb_h$ and $\mb_a$ denote the projections to $\Z^h$ and $\Z^a$. For $\mb_h\in \Z^h$, define
$$
f_{\mb_h}(\tb_h)=\frac{1}{t_{\mb_h}}\sum_{w\in \W} e^{2\pi i\mb_h \cdot w(\tb_h)}\,,
$$
where $t_{\mb_h}$ denotes the size of the stabilizer of $\mb_h$ under the action of $\W$.
If $\Gamma_{\nb_h}$ denotes the representation of highest weight~$\nb_h$, then
$$
\Tr(\Gamma_{\nb_h}(\tb_h))= \sum_{\mb_h\preceq \nb_h} m^{\mb_h}_{\nb_h} f_{\mb_h}(\tb_h)\,,
$$
where the sum runs over weights $\mb_h\in \Cc$. Equivalently, we have
\begin{equation}\label{equation: dec}
f_{\mb_h}(\tb_h)=\sum_{\nb_h\preceq \mb_h} d_{\mb_h}^{\nb_h}\Tr(\Gamma_{\nb_h}(\tb_h))\,.
\end{equation}
We remark that Proposition \ref{proposition: Gupta} ensures that, for each $\mb_h$, the number of nonzero coefficients $d_{\mb_h}^{\nb_h}$ in the above equation, as well as the size of each of them, is $O(1)$. By taking the Fourier expansion of $D$, we obtain
$$
F(\tb)= \frac{1}{\#\W}\sum_{\mb\in \Z^q} c_{\mb} t_{\mb_h}f_{\mb_h}(\tb_h)e^{2\pi i \mb_a\cdot \tb_a} =  \frac{1}{\#\W}\sum_{\mb\in \Cc\times \Z^a}\left( \sum_{w\in\W}c_{w(\mb)}\right) t_{\mb_h}f_{\mb_h}(\tb_h)e^{2\pi i \mb_a\cdot \tb_a}\,.
$$
Let $M\geq 1$ be a positive integer (arbitrary for the moment and to be determined later). Let $\Cc^{\leq M}$ denote the subset of $\Cc\times \Z^a$ made of weights $\mb$ whose components have absolute value $\leq M$. Note that if $\mb \in \Cc^{\leq M}$, then in particular we have that $||\mb_h||_{\fund}\leq M$. Let $\Cc^{>M}$ denote the complement of $\Cc^{\leq M}$ in $\Cc\times \Z^a$. 

On the one hand, by invoking the bounds from part iii) of Proposition~\ref{proposition: vinogradov}, we have
\begin{equation}\label{equation: tail}
\begin{array}{lll}
F_{> M}(\tb)&:=&\displaystyle{\frac{1}{\#\W}\sum_{\mb\in \Cc^{>M} }\left( \sum_{w\in\W}c_{w(\mb)}\right) t_{\mb_h}f_{\mb_h}(\tb_h)e^{2\pi i \mb_a\cdot \tb_a}}\\[8pt]
& = & \displaystyle{O\left( \sum_{m> M}m^{q-1}\frac{1}{m}\left(\frac{rK\sqrt q}{2\pi m\Delta}\right)^r\right)= O\left( \frac{1}{M^{r-q+1}\Delta^r}\left(\frac{rK\sqrt q}{2\pi}\right)^r\right)}\,.
\end{array}
\end{equation}
On the other hand, consider the class function
\begin{equation}\label{equation: head}
\begin{array}{lll}
F_{\leq M}(\tb) & := & \displaystyle{\frac{1}{\#\W}\sum_{\mb\in\Cc^{\leq M}}\left( \sum_{w\in\W}c_{w(\mb)}\right) t_{\mb_h}f_{\mb_h}(\tb_h)}e^{2\pi i \mb_a\cdot \tb_a}\\[8pt] 
&= & \displaystyle{\delta(F_{\leq M}(\tb)) +  \frac{1}{\#\W}\sum_{\mb\in\Cc^{\leq M}}\left( \sum_{w\in\W}c_{w(\mb)}\right) t_{\mb_h}\sum_{\oub\not = \nb_h\preceq \mb_h} d_{\mb_h}^{\nb_h}\Tr(\Gamma_{\nb_h}(\tb_h))e^{2\pi i \mb_a\cdot \tb_a}\,.}
\end{array}
\end{equation}
In the above expression $\delta(F_{\leq M}(\tb))$ stands for the multiplicity of the identity representation in $F_{\leq M}(\tb)$. Note that $F_{\leq M}$ is a finite linear combination of irreducible characters of $\ST(A)$ and that by Proposition~\ref{proposition: Gupta} we may assume that $M$ is large enough so that $\delta(F_{\leq M}(\tb))=\delta(F(\tb))$, which we will do from now on. 

The next step is to bound the virtual dimension of the nontrivial part of $F_{\leq M}(\tb)$ in order to be able to apply Proposition~\ref{proposition: Murtyes}. More precisely, if $p_{\nb_h}$ denotes the coefficient multiplying $\Tr(\Gamma_{\nb_h}(\tb_h))$ in \eqref{equation: head}, then we have 
$$
\sum_{\mb\in \Cc^{\leq M}}\sum_{\oub\not =\nb_h\preceq \mb_h}|p_{\nb_h}| \dim(\Gamma_{\nb_h})= O\left( \sum_{\oub\not =\mb_h\in\Cc^{\leq M}}c_{\mb_h} \dim(\Gamma_{\mb_h})\right) =  O\left( \sum_{0<m\leq M}m^{q-1}\frac{1}{m}\left(\frac{rK\sqrt q}{2\pi m\Delta}\right)^\rho m^{\varphi}\right).
$$
In the above computation we have used: Proposition~\ref{proposition: Gupta} to bound the size and number of nonzero entries in the inverse of the matrix of weight multiplicities; Proposition~\ref{proposition: dimbound} to control the dimension of the representations of weight lower than a given one and to bound the dimension of the representation $\Gamma_{\mb_h}$ in terms of $||\mb_h||_{\fund}$; and part iii) of Proposition~\ref{proposition: vinogradov} to bound the Fourier coefficients for an unspecified (for the moment) $1\leq \rho\leq r$. We will now distinguish two cases, depending on whether $\varphi$ is zero or not.

Suppose first that $\varphi$ is nonzero. Take $r=\rho=q+\varphi-1$, which we note that satisfies $r\geq 1$. Then
\begin{equation}\label{equation: virtdimbound}
\sum_{\mb\in \Cc^{\leq M}}\sum_{\oub\not =\nb_h\preceq \mb_h}|p_{\nb_h}| \dim(\Gamma_{\nb_h})=O\left( \sum_{0<m\leq M}\frac{1}{m\Delta^{q+\varphi-1}}\right)=O\left( \frac{\log(M)}{\Delta^{q+\varphi-1}}\right)\,.
\end{equation}
Let $L>0$ be the implied constant in the bound of part iii) of Proposition \ref{proposition: dimbound} for the motivic weight, so that for $\mb_h\in \Cc^{\leq M}$, we have $w_{\Gamma_{\mb_h}}\leq LM$. 
Using the decomposition 
$$
F(\tb)=F_{\leq M}(\tb)+F_{>M}(\tb)\,,
$$
the tail \eqref{equation: tail} and virtual dimension \eqref{equation: virtdimbound} bounds, and applying Proposition~\ref{proposition: Murtyes}, we obtain
\begin{equation}\label{equation: sumaprox}
\sum_{\NN(\p)\leq x} F(\tb_\p)=\delta(F(\tb))\Li(x)+ O\left(\frac{\log(M)}{\Delta^{q+\varphi-1}}\sqrt x\log\big(N(x+LM)\big)\right)+O\left(\frac{\Li(x)}{M^{\varphi}\Delta^{q+\varphi-1}}\right)\,.
\end{equation}
It follows from the proof of Lemma \ref{lemma: dF} that 
\begin{equation}\label{equation: quarra}
\delta(F(\tb))=\mu(I)+O(\Delta)\,.
\end{equation}
Therefore, to conclude the proof, it will suffice to balance the error terms in \eqref{equation: sumaprox} with $O(\Delta \Li(x))$. In the case that $\varphi$ is nonzero, we may take
\begin{equation}\label{equation: choiceDelteffst}
\Delta:=x^{-\varepsilon}\log(x)^{4\varepsilon}\log(Nx)^{2\varepsilon}\,,\qquad M=\left\lceil \Delta^{-\frac{q+\varphi}{\varphi}}\right\rceil\,,
\end{equation}
where $\varepsilon=\varepsilon_\g$ is as defined in \eqref{equation: epsilondefinition}. In view of Lemma \ref{lemma: dF}, this concludes the proof, provided that we verify that this choice of $\Delta$ satisfies the constraint \eqref{equation: Deltaconstraint}. This amounts to $2\Delta\leq |I|$, or equivalently to
$$
x\geq  \frac{2^{\varepsilon^{-1}}}{|I|^{\varepsilon^{-1}}}\log(x)^4\log(Nx)^2\,.
$$
By the elementary Lemma~\ref{lemma: auxcomp} below, this is easily seen to be the case as soon as $x\geq x_0$, where 
\begin{equation}\label{equation: x0lowerbound}
x_0=O\left(\nu_\g(|I|)\log(2N)^2\log(\log(4N))^4 \right)\,.
\end{equation}
Suppose now that $\varphi=0$. We take $r=q$ and use the tail bound \eqref{equation: tail} as in the previous case. To bound the Fourier coefficients of $F_{\leq M}(\tb)$, we use the bound 
$|c_m|=O(1/ m)$ 
if $q=1$ and the bound corresponding to $\rho=q-1$ otherwise (as in part iii) of Proposition~\ref{proposition: vinogradov}). We obtain
\begin{equation}\label{equation: forzerophi}
\sum_{\NN(\p)\leq x} F(\tb_\p)=\delta(F(\tb))\Li(x)+ O\left(\frac{\log(M)}{\Delta^{q-1}}\sqrt x\log\big(N(x+LM)\big)\right)+O\left(\frac{\Li(x)}{M\Delta^{q}}\right)\,.
\end{equation}
To balance the error terms in the above equation with $O(\Delta \Li(x))$, we may take
$$
\Delta:=x^{-1/(2q)}\log(x)^{2/q}\log(Nx)^{1/q}\,,\qquad M=\left\lceil \Delta^{-q-1}\right\rceil\,.
$$
Since $\varepsilon_\g=1/(2q)$ in this case, this yields precisely the error term of the statement of the theorem. Again,~$\Delta$ satisfies the constraint ~\eqref{equation: Deltaconstraint} as soon as $x\geq x_0$, where $x_0$ is as in \eqref{equation: x0lowerbound}.
\end{proof}

We leave the proof of the following to the reader.
\begin{lemma}\label{lemma: auxcomp}
For integers $r,N\geq 1$, with $r$ even, and a real number number $A>0$, we have that
$$
A\log(x)^r\log(Nx)^2<x
$$
provided that $x>C\log(2N)^2\log(\log(4N))^{r}\max\{1,A\log(A)^{r+2}\}$ for some $C>0$ depending exclusively on~$r$. 
\end{lemma}

\begin{remark}
To simplify the statement of Theorem~\ref{theorem: effectivest}, we have assumed Conjecture~\ref{conjecture: GRH} for every irreducible character $\chi$ of $\ST(A)$. It is however clear from the proof that this hypothesis can be relaxed: it suffices to assume Conjecture \ref{conjecture: GRH} for those representations $\Gamma_{\mb}$ with $\mb\in \Cc^{\leq M}$, where $M$ is as in \eqref{equation: choiceDelteffst}. 
\end{remark}

\begin{remark}
The choice of the exponent of $x$ in the error term in Theorem~\ref{theorem: effectivest}
is dictated by the balancing of $O(\Delta \Li(x))$ with the first of the two error terms in \eqref{equation: sumaprox}. The balancing with the second error term only affects the logarithmic factors.
\end{remark}

\section{Applications}\label{section: Applications}

In this section we discuss three applications of Theorem~\ref{theorem: effectivest}. In \S\ref{section: intlin} we consider an interval variant of Linnik's problem for abelian varieties. Given an abelian variety $A$ defined over $k$ of dimension $g$ and a subinterval $I$ of $[-2g,2g]$, this asks for an upper bound on the least norm of a prime $\p$ not dividing $N$ such that the normalized Frobenius trace $\overline a_\p(A)$ lies in $I$. In \S\ref{section: signlin} we consider a sign variant of Linnik's problem for a pair of abelian varieties $A$ and $A'$ defined over the number field~$k$ and such that $\ST(A\times A') \simeq \ST(A) \times \ST(A')$. This asks for an upper bound on the least norm of a prime $\p$ such that $a_\p(A)$ and $a_\p(A')$ are nonnegative and have opposite sign. Finally, in \S\ref{section: maxnumpoints}, when $A$ is an elliptic curve with CM, we conditionally determine (up to constant multiplication) the asymptotic number of primes for which $a_\p(A)= \lfloor2\sqrt {\NN(\p)}\rfloor$.

While \S\ref{section: intlin} is a direct consequence of Theorem~\ref{theorem: effectivest}, both \S\ref{section: signlin} and \S\ref{section: maxnumpoints} require slight variations of it. We will explain how to modify the proof of Theorem \ref{theorem: effectivest} to obtain these versions.

\subsection{Interval variant of Linnik's problem for abelian varieties}\label{section: intlin}

Theorem \ref{theorem: effectivest} has the following immediate corollary.

\begin{corollary}\label{corollary: intlin}
Assume the hypotheses and notations of Theorem \ref{theorem: effectivest}. For every nonempty subinterval $I$ of $[-2g,2g]$, there exists a prime $\p$ not dividing $N$ with
$$
\NN(\p)=O\left(\nu_\g(\min\{|I|,\mu(I)\})\log(2N)^2\log(\log(4N))^4\right)
$$
such that $\overline a_\p\in I$. 
\end{corollary}

\begin{proof}
There exist constants $K_1,K_2>0$ such that, for $x\geq K_2\nu_\g(|I|)\log(2N)^2\log(\log(4N))^4$, the number of primes $\p$ such that $\NN(\p)\leq x$ and $\overline a_\p\in I$ is at least
$$
\mu(I)\Li(x)\left(1-\frac{K_1}{\mu(I)}\Delta\right)\,,
$$ 
where $\Delta$ is as in (\ref{equation: choiceDelteffst}). This count will be positive provided that $K_1\Delta<\mu(I)$, or equivalently if
$$
x> \frac{K_1^{\varepsilon^{-1}}}{\mu(I)^{\varepsilon^{-1}}}\log(x)^4\log(Nx)^2\,.
$$
One easily verifies that this condition is satisfied for $x\geq x_0$, for some $x_0=O(\nu_\g(\mu(I))\log(2N)^2\log(\log(4N))^4)$, and the corollary follows.
\end{proof}

\subsection{Frobenius sign separation for pairs of abelian varieties}\label{section: signlin}

In this section we will provide an answer to the Frobenius sign separation problem for pairs of abelian varieties using a variation of Theorem~\ref{theorem: effectivest}. Resume the notations of \S\ref{section: proof}; additionally, let $A'$ be an abelian variety defined over $k$ and let $g'$, $N'$, $\mu'$, etc, denote the correponding notions. 
We will make the hypothesis that the natural inclusion of $\ST(A\times A')$ in the product $\ST(A)\times \ST(A')$ is an isomorphism. 
\begin{hypothesis}\label{hypothesis: prod ST}
We have that $\ST(A\times A')\simeq \ST(A)\times \ST(A')$.
\end{hypothesis}
Theorem \ref{theorem: linniksign} shows that under the conjectures of \S\ref{section: conjectures}, this hypothesis ensures the existence of a prime $\p$ not dividing $NN'$ such that 
\begin{equation}\label{equation: opsign}
a_\p(A)\cdot a_\p(A')<0
\end{equation}
and, in fact, determines the asymptotic density of such primes. Corollary \ref{corollary: linniksign}, which gives an upper bound on the least norm of such a prime, is then an immediate consequence.
Note that requiring $A$ and $A'$ not to be isogenous does not guarantee the existence of a prime satisfying \eqref{equation: opsign}, as it is shown by the trivial example in which~$A'$ is taken to be a proper power of $A$. 

Write the complexified Lie algebra of $\ST(A)$ (resp. $\ST(A')$) as $\g=\s\times \fa$ (resp. $\g'=\s'\times \fa'$), where $\s,\s'$ are semisimple and $\fa,\fa'$ are abelian. Throughout this section, write 
\begin{equation}\label{equation: epsilondefinition2}
\varepsilon_{\g,\g'}:=\frac{1}{2(q+q'+\varphi+\varphi'-1)}\,,
\end{equation}
where $\varphi$ (resp. $\varphi'$) is the size of the set of positive roots of~$\s$ (resp.~$\s'$) and~$q$ (resp.~$q'$) is the rank of~$\g$ (resp.~$\g'$). Define
$$
\nu_{\g,\g'}\colon \R_{>0}\rightarrow \R_{>0}\,,\qquad
\nu_{\g,\g'}(z)=\max\left\{1, \frac{\log(z)^8}{z^{1/\varepsilon_{\g,\g'}}}\right\}
$$
\begin{theorem}\label{theorem: linniksign}
Let $k$ be a number field, and let $g$ and $g'$ positive integers $\geq 1$. Let $A$ (resp. $A'$) be an abelian variety defined over $k$ of dimension $g$ (resp. $g'$), absolute conductor $N$ (resp. $N'$), and such that $\ST(A)$ (resp. $\ST(A')$) is connected. Assume that Hypothesis \ref{hypothesis: prod ST} holds. Suppose that the Mumford--Tate conjecture holds for $A\times A'$, and that Conjecture~\ref{conjecture: GRH} holds for every product $\chi\cdot \chi'$ of irreducible characters $\chi$ of $\ST(A)$ and $\chi'$ of $\ST(A')$.  
For each prime $\p$ not dividing~$NN'$, let $\overline a_{\p}$ (resp. $\overline a_\p'$) denote the normalized Frobenius trace of $A$ (resp. $A'$) at $\p$. Then for all nonempty subintervals $I$ of $[-2g,2g]$ and $I'$ of $[-2g',2g']$, we have
$$
\sum_{\NN(\p)\leq x}\delta_{I}(\overline a_{\p})\delta_{I'}(\overline a_\p')=\mu(I)\mu'(I')\Li(x)+O\left(\frac{x^{1-\varepsilon_{\g,\g'}}\log(NN'x)^{2\varepsilon_{\g,g'}}}{\log (x)^{1-6\varepsilon_{\g,\g'}}}\right)\qquad \text{for $x\geq x_0$,}
$$
where $x_0=O\left(\nu_{\g,\g'}(\min\{|I|,|I'|\})\log(2NN')^2\log(4NN')^6\right)$.
\end{theorem}

\begin{proof} Let $(\alpha,\beta)$ and $(\alpha',\beta')$ denote the interiors of $I$ and $I'$, respectively. For a common choice of $\Delta>0$, define $F_{\Delta,I}(\tb)$ and $F'_{\Delta,I'}(\tb')$ relative to undetermined positive integers $r$ and $r'$ in a manner analogous to \eqref{equation: defofF}. 
Let $M\geq 1$ be a positive integer (arbitrary for the moment and to be determined later). In analogy with the definition of $L>0$ in the line following \eqref{equation: virtdimbound}, let $L'>0$ be the implied constant in the bound $w_{\Gamma_{\mb'_h}}=O(||\mb'_h||_{\fund})$. Let $L''$ denote $\max\{L,L'\}$.

Suppose that $\varphi+\varphi'$ is nonzero. Choose $r=q+\varphi-1$ and $r'=q'+\varphi'-1$. Analogues of \eqref{equation: tail} and \eqref{equation: virtdimbound} give
\begin{equation}\label{equation: double aprox}
\begin{array}{lll}
\displaystyle{\sum_{\NN(\p)\leq x} F(\tb_\p)F'(\tb_\p')} & = & \delta(F(\tb))\delta(F'(\tb))\Li(x)\\[8pt]
& & \displaystyle{+\, O\left(\frac{\log(M)^2}{\Delta^{q+q'+\varphi+\varphi'-2}}\sqrt x\log\big(NN'(x+L''M)\big)\right)}\\[8pt]
&  & \displaystyle{+\, O\left(\frac{\Li(x)}{M^{\varphi+\varphi'}\Delta^{q+q'+\varphi+\varphi'-2}}\right)\,.}
\end{array}
\end{equation}
Here we have used that the multiplicity of the trivial representation $\delta(\Gamma_{\mb_h}\otimes\Gamma_{\mb'_h})$ is zero unless both $\mb_h$ and $\mb'_h$ are $\oub$, as follows from Hypothesis \ref{hypothesis: prod ST}. We also used that the conductor of $A\times A'$ is $O(NN')$. By the proof of Lemma \ref{lemma: dF}, we have 
$$
\delta(F(\tb)F'(\tb'))=\delta(F(\tb))\delta(F'(\tb'))=\mu(I)\mu'(I')+O(\Delta)\,.
$$ 
If $\varphi+\varphi'$ is nonzero, take $\varepsilon:=\varepsilon_{\g,\g'}$ as in \eqref{equation: epsilondefinition2} and
\begin{equation}\label{equation: valdeltadouble}
\Delta:=x^{-\varepsilon}\log(x)^{6\varepsilon}\log(NN'x)^{2\varepsilon}\,,\qquad M=\left\lceil \Delta^{-\frac{q+q'+\varphi+\varphi'-1}{\varphi+\varphi'}}\right\rceil\,,
\end{equation}
which balance the error terms in \eqref{equation: double aprox} with $O(\Delta\Li(x))$.

Suppose now that $\varphi=\varphi'=0$. Choose $r=q$ and $r'=q'$. As in \eqref{equation: forzerophi}, we apply part iii) of Proposition~\ref{proposition: vinogradov} with $\rho=r$ (resp. $\rho'=r'$) to bound the Fourier coefficients of $F_{> M}$ (resp. $F'_{> M}$); for the Fourier coefficients of $F_{\geq M}$ (resp. $F'_{\geq M}$) we use the bound $c_m=O(1/m)$ (resp. $c'_m=O(1/m)$) if $q=1$ (resp. $q'=1$) and the bound corresponding to $\rho=q-1$ (resp. $\rho'=q'-1$) if $q>1$ (resp. $q'>1$). We obtain
$$
\sum_{\NN(\p)\leq x} F(\tb_\p)F'(\tb_\p')=  \delta(F(\tb))\delta(F'(\tb))\Li(x)+ O\left(\frac{\log(M)^2}{\Delta^{q+q-2}}\sqrt x\log\big(NN'(x+L''M)\big)\right)+ O\left(\frac{\Li(x)}{M^{2}\Delta^{q+q'}}\right)\,. 
$$
In order to balance the error terms of the above expression with $O(\Delta\Li(x))$, we take
$$
\Delta:=x^{-1/(q+q'-1)}\log(x)^{3/(q+q'-1)}\log(NN'x)^{1/(q+q'-1)}\,,\qquad M=\left\lceil \Delta^{-\frac{q+q'-1}{2}}\right\rceil\,.
$$
This yields the error term in the statement of the theorem, since $\varepsilon=1/2(q+q'-1)$ when $\varphi=\varphi'=0$.
It only remains to determine the set of $x$ for which the constraint $2\Delta\leq \min\{|I|,|I'|\}$, or equivalently the inequality
$$
x> \frac{2^{\varepsilon^{-1}}}{\min\{|I|,|I'|\}^{\varepsilon^{-1}}}\log(x)^6\log(Nx)^2\,,
$$
is satisfied. As follows from Lemma~\ref{lemma: auxcomp}, this happens if $x\geq x_0$, where $x_0$ is as in the statement of the theorem.
\end{proof}

\begin{corollary}\label{corollary: linniksign} 
Assume the hypotheses of Theorem \ref{theorem: linniksign}. Then there exists a prime $\p$ not dividing $NN'$ with 
$$
\NN(\p)=O\left(\log(2NN')^2\log(\log(4NN'))^6 \right)
$$ 
such that $a_\p(A)$ and $a_\p(A')$ are nonzero and of opposite sign.
\end{corollary}

\begin{proof}
In Theorem \ref{theorem: linniksign}, take the subintervals $I=(\delta,2g-\delta)$ and $I'=(-2g+\delta,-\delta)$ for $\delta=1/2$. There exist constants $K_1,K_2,K_3>0$ such that, for $x\geq K_3\log(2NN')^2\log(\log(4NN'))^6$, the number of primes $\p$ such that $\NN(\p)\leq x$, $\overline a_\p\in I$, and $\overline a_\p'\in I'$ is at least
$$
K_1\Li(x)(1-K_2\Delta)\,,
$$ 
where $\Delta$ as in (\ref{equation: valdeltadouble}). This count will be positive provided that $K_2\Delta<1$, or equivalently if
$$
x> K_2^{\varepsilon^{-1}}\log(x)^6\log(NN'x)^2\,.
$$
One easily verifies that this condition is satisfied for $x\geq x_0$, for some $x_0=O\left(\log(2NN')^2\log(\log(4NN'))^6\right)$.
\end{proof}

\begin{remark}
Under the current assumption that $\ST(A)$ and $\ST(A')$ are connected, one may wonder when is Hypothesis \ref{hypothesis: prod ST} satisfied. According to \cite[Lem. 6.10]{BK15a} this should happen rather often when $\Hom(A_\Qbar,A'_\Qbar)=0$. More precisely, if both $A$ and $A'$ satisfy the Mumford--Tate conjecture, $\Hom(A_\Qbar,A'_\Qbar)=0$, $A$ has no factors of type $\mathrm{IV}$, and either: 
\begin{enumerate}[i)]
\item $A'$ is of CM type; or
\item $A'$ has no factors of type $\mathrm{IV}$; 
\end{enumerate}
then Hypothesis \eqref{hypothesis: prod ST} holds.
\end{remark}

\subsection{CM elliptic curve reductions with maximal number of points}\label{section: maxnumpoints}

In this section we prove a variation of Theorem \ref{theorem: effectivest} in a situation where the interval $I$ \emph{varies with $x$}.
We determine (up to constant multiplication and under the assumption of Conjecture~\ref{conjecture: GRH}) the number of primes at which the Frobenius trace of an elliptic curve defined over $k$ with potential CM achieves the integral part of the Weil bound. We will start by assuming that $A$ has CM already defined over~$k$, that is, that $\ST(A)\simeq \Unitary(1)$. 

Throughout this section let $x\geq 2$ and $y\geq 2^{2/3}$ be real numbers. Let $I_{y}$ denote the subinterval $[2-y^{-1/2},2]$ of $[-2,2]$.  

\begin{lemma}\label{lemma: measureint}
For $\mu=dz/(\pi\sqrt{4-z^2})$, we have 
$$
\mu(I_{y})= \frac{1}{\pi y^{1/4}} +O\left(\frac{1}{y^{3/4}}\right)\qquad \text{for every $y\geq 2^{2/3}$}\,.
$$
\end{lemma}

\begin{proof}
Recall the map from \eqref{equation: mapT}, which in this case is simply
$$
T\colon \R\rightarrow [-2,2]\,,\qquad T(\theta)=2\cos(2\pi\theta)\,.
$$
We first determine the preimage $[-\theta_y,\theta_y]:=T^{-1}(I_{y})\cap [-1/2,1/2]$. We easily find
$$
\theta_y=\frac{1}{2\pi}\arccos\left( 1-\frac{y^{-1/2}}{2}\right)=\frac{1}{2\pi y^{1/4}}+O\left(\frac{1}{y^{3/4}}\right)\qquad \text{for every $y\geq 2^ {2/3}$.}
$$ 

Since $\mu$ is the pushforward via $T$ of the uniform measure on $[0,1]$, we have that $\mu(I_{y})$ is the length of $[-\theta_y,\theta_y]$, from which the lemma follows. 
\end{proof}

\begin{proposition}\label{proposition: reductions}
Let $A$ be an elliptic curve with CM defined over $k$ of absolute conductor $N$. Suppose that Conjecture \ref{conjecture: GRH} holds for every character\footnote{In other words, we assume that GRH holds for the Hecke $L$-function attached to every integral power of the Grossencharacter attached to $A$.} of $\ST(A)\simeq \Unitary(1)$. For each prime $\p$ not dividing $N$, let $\overline a_{\p}$ denote the normalized Frobenius trace of $A$ at $\p$. For every $x\geq 2$, we have
$$
\sum_{\NN(\p)\leq x} \delta_{I_{y}}(\overline a_\p)=\frac{1}{\pi y^{1/4}}\Li(x)+O\left( \sqrt x\log(Nx)\log(x) \right)\qquad \text{for every $x^{2/3}\leq y\leq x$}\,.
$$
\end{proposition}

\begin{proof}
Let us start by choosing $\Delta=y^{-1/2-\nu}$, for some $\nu>0$ so that hypothesis \eqref{equation: Deltaconstraint} for $\Delta$ and $I_{y}$ is satisfied. Let us choose the function $D=D_{\Delta,I_{y}}$ from Proposition \ref{proposition: vinogradov} relative to $r=1$. Proceeding exactly as in case $\varphi=0$ of the proof of Theorem \ref{theorem: effectivest} we arrive at \eqref{equation: forzerophi} (the fact that the exponent of $\Delta$ in the mid error term of \eqref{equation: forzerophi} is $q-1$ is precisely what makes the case $q=1$ special: this allowed us to choose~$\Delta$ beforehand and arbitrarily small).
As seen in the proof of Lemma \ref{lemma: dF}, we have $\delta(F(\theta))=\mu(I_{y})+O(\Delta)$. Then, the choice of $M=\Delta^{-2}$ gives 
$$
\sum_{\NN(\p)\leq x}F(\theta_\p)=\mu(I_{y})\Li(x)+ O\left(\sqrt x\log(Nx)\log(x)\right)\,.
$$
By Lemma \ref{lemma: dF} and Lemma \ref{lemma: measureint} we have that
$$
\sum_{\NN(\p)\leq x}\delta_{I_y}(\overline a_\p)=\frac{1}{2\pi y^{1/4}}\Li(x)+O\left(\frac{\Li(x)}{y^ {3/4}}\right)+ O\left(\sqrt x\log(Nx)\log(x)\right)\,.
$$
The proposition now follows from the fact that if $y\geq x^ {2/3}$, then the error term $O(\Li(x)/y^ {3/4})$ is subsumed in the error term of the statement.
\end{proof}
For every $x\geq 2$, define
$$
R(x):=\{\p\nmid N\text{ prime of }k\colon \NN(\p)\leq x\text{ and }|\overline a_\p-2|< x^{-1/2}\}\,,
$$
and for every $x^ {2/3}\leq y\leq x$, define
$$
S(y,x):=\{\p\nmid N\text{ prime of }k\colon y< \NN(\p)\leq x\text{ and }|\overline a_\p-2|< y^{-1/2}\}.
$$
Lemma \ref{lemma: measureint} and Proposition \ref{proposition: reductions} have the following corollary.

\begin{corollary}\label{corollary: infseq}
Assume the same hypotheses as in Proposition~\ref{proposition: reductions}. 
For every $x\geq 2$, we have:
\begin{enumerate}[i)]
\item $\# R(x)=\frac{1}{\pi x^{1/4}}  \Li(x)+O(\sqrt x\log(Nx)\log(x))$.
\item $\# S(y,x)=\frac{1}{\pi y^{1/4}}  \left(\Li(x)-\Li(y)\right)+O(\sqrt x\log(Nx)\log(x))$ for every $x^{2/3}\leq y\leq x$.
\end{enumerate}
\end{corollary}

Let $M_k(x)$ denote the set of primes $\p$ of $k$ not dividing $N$ with $ \NN(\p)\leq x$ such that $a_\p=\lfloor 2\sqrt {\NN(\p)}\rfloor$, or equivalently such that $|\overline a_\p-2|<1/\sqrt{\NN(\p)}$. Let $2<x_n<x_{n-1}<\dots <x_2<x_1=x$ be real numbers. Note that
\begin{equation}\label{equation: partition}
R(x)\subseteq M_k(x)\subseteq \bigcup_{j=1}^{n-1}S(x_{j+1},x_j) \cup \{\p\nmid N\text{ prime of }k\colon \NN(\p)\leq x_n \}\,.
\end{equation}

\begin{proposition}\label{proposition: asympmax}
Assume the same hypotheses as in Proposition~\ref{proposition: reductions}. Then
$$
\#M_k(x)\asymp_N  \frac{x^{3/4}}{\log(x)}\qquad \text{as }x\rightarrow \infty\,,
$$
\end{proposition}

\begin{proof}
From \eqref{equation: partition} and Corollary \ref{corollary: infseq}, we immediately obtain $x^{3/4}/\log(x)=O_N(\#M_k(x))$. To show that $\#M_k(x)=O_N( x^{3/4}/\log(x))$, for $j=1,\dots, n:=\lfloor x^{1/16}\rfloor$, define $x_j:=x/j^4$. 
Since $x_n=O(x^{3/4})$, by \eqref{equation: partition} and Corollary \ref{corollary: infseq}, we have
$$
\begin{array}{lll}
\#M_k(x) & \leq & \displaystyle{\sum_{j= 1}^{n-1} \frac{j+1}{ \pi x^{1/4}}\left(\Li(x_{j})-\Li(x_{j+1}) \right)+\Li(x_n)+O(x^{1/2+1/16}\log(Nx)\log(x))}\\[6pt]
& = & \displaystyle{\frac{ 2}{\pi x^{1/4}}\Li(x)+\frac{1}{ \pi x^{1/4}}\sum_{j=2}^{n-1}\Li(x_j) }- \frac{n}{\pi x^{1/4}}\Li(x_n)+ O_N\left(\frac{x^{3/4}}{\log(x)}\right).  
\end{array}
$$
In view of the above, the proposition will follow from the fact that
$$
\sum_{j=2}^{n} \frac{x^{3/4}}{j^4\log\left( \frac{x}{j^4}\right)}=O\left(\frac{x^{3/4}}{\log(x)}\right)\,.
$$
But the change of variable $z=x/y^4$ gives
$$
\sum_{j=2}^{n} \frac{x^{3/4}}{j^4\log\left( \frac{x}{j^4}\right)}=O\left(\int_2^{x^{1/16}} \frac{x^{3/4}}{y^4 \log\left(x/y^4\right)}dy\right)=O\left( \int_{x^{3/4}}^{x/16} \frac{1}{z^{1/4}\log(z)}dz\right)\,.
$$
Set $f(z)=1/(z^{1/4}\log(z))$ and $\theta(z)=4 z^{3/4}/3$, so that integration by parts yields
$$
F(x):=\int_{x^{3/4}}^{x/16} f(z)dz = \frac{\theta(x/16)}{\log(x/16)}-\frac{\theta(x^{3/4})}{\log(x^{3/4})}+ \int_{x^{3/4}}^{x/16}\frac{\theta(z)}{z\log^2(z)}dz\,.  
$$
Since the first term in the right-hand side of the above equation is $O(x^{3/4}/\log(x))$, and the second term is bounded by $F(x)/\log(x^{3/4})$, we deduce that $F(x)=O(x^{3/4}/\log(x))$, which concludes the proof. 
\end{proof}

\begin{corollary}\label{corollary: asympmax}
Let $A$ be an elliptic curve with potential CM (say by an imaginary quadratic field $K$) not defined over $k$. Under Conjecture \ref{conjecture: GRH} for every character of $\ST(A_{kK})\simeq \Unitary(1)$, we have
$$
\#M_k(x)\asymp_N  \frac{x^{3/4}}{\log(x)}\qquad \text{as }x\rightarrow \infty\,.
$$
\end{corollary}

\begin{proof}
Consider the base change $A_{kK}$ and the set of primes of $kK$ defined as
$$
M_{kK}^{\mathrm{split}}(x):=\{\PP\nmid N\text{ prime of $kK$ split over }k\colon \NN(\PP)\leq x \text{ and } a_{\PP}(A_{kK})=\lfloor 2\sqrt{\NN(\PP)}\rfloor\}\,.
$$  
Since the number of primes of $kK$ nonsplit over $k$ of norm up to $x$ is $O(\sqrt x)$, in view of Proposition \ref{proposition: asympmax}, we have that
$$
\# M_{kK}(x) \sim \# M_{kK}^{\mathrm{split}}(x) \qquad \text{as }x\rightarrow \infty\,. 
$$
On the other hand, the map
$$
M_{kK}^{\mathrm{split}}(x)\rightarrow M_k(x)\,,\qquad \PP\mapsto \PP\cap k
$$
is 2 to 1, and we thus get
$$
\#M_k(x)\sim \frac{1}{2}\# M_{kK}^{\mathrm{split}}(x)\qquad \text{as }x\rightarrow \infty\,. 
$$
\end{proof}

As noted in the introduction, it was shown unconditionally by James and Pollack \cite[Theorem~1]{JP17}
that
$$
\#M_k(x)\sim  \frac{2}{3\pi}\frac{x^{3/4}}{\log(x)}\qquad \text{as }x\rightarrow \infty.
$$
That result, which gives a partial answer to a question of Serre \cite[Chap. II, Question 6.7]{Ser20},
builds on a conditional result of James et al. \cite{JTTWZ16}; that result is similar to ours,
except that it aggregates primes for which the Frobenius trace is extremal in both directions. 
The added ingredient in \cite{JP17} is the use of unconditional estimates for the number of primes in an imaginary quadratic field lying in a sector; such an estimate has been given by Maknys \cite{Mak83}, modulo a correction described in \cite{JP17}. (For the Gaussian integers, see also \cite{Zar91}.)

\begin{remark}\label{remark: dimensionsST}
Let $A$ be an abelian variety of dimension $g\geq 1$ defined over $k$. Let $d$ denote the real dimension of $\ST(A)$. It follows from \cite[\S8.4.4.4]{Ser12} that
\begin{equation}\label{equation: maintermcount}
\mu([2g-x^{-1/2},2g])\cdot\Li(x)\sim C\cdot\frac{ x^ {1-d/4}}{\log(x)}\qquad \text{as $x\rightarrow \infty$,}
\end{equation}
for some constant $C>0$. When $d>2$, the count \eqref{equation: maintermcount} is subsumed in the error term of Theorem \ref{theorem: effectivest}. There is thus no hope that the method of proof of Corollary \ref{corollary: infseq} can be extended to the case $d>2$ to obtain the analogue statement. 

When $d=1$ (in which case $A$ is $\Qbar$-isogenous to the power of a CM elliptic curve and $\ST(A)\simeq \Unitary(1)$), it is not difficult to generalize Proposition \ref{proposition: reductions} to show that the number of primes $\p$ such that $a_\p(A)=\lfloor 2g\sqrt{\NN(\p)} \rfloor$ is again $\asymp_N x^{3/4}/\log(x)$. Note that for these primes, the equality $\lfloor 2g\sqrt{\NN(\p)}\rfloor=g\lfloor 2\sqrt{\NN(\p)}\rfloor$ needs to hold because of the Weil-Serre bound.

As Andrew Sutherland kindly explained to us, when $d=2$ there are already examples of abelian surfaces~$A$ defined over $\Q$ for which there are no primes $p$ of good reduction for $A$ such that 
\begin{equation}\label{equation: pattainmaxbound}
a_p(A)=2\lfloor 2\sqrt {p}\rfloor\,.
\end{equation} 
Indeed, let $A$ be the product of two elliptic curves $E_1$ and $E_2$ defined over $\Q$ with CM by two nonisomorphic imaginary quadratic fields $M_1$ and $M_2$, respectively. Suppose there were a prime $p>3$ satisfying \eqref{equation: pattainmaxbound} of good reduction for $A$. Then $a_p(E_1)=a_p(E_2)=\lfloor 2\sqrt {p}\rfloor$ and $p$ would be ordinary for both $E_1$ and $E_2$. This would force both $M_1$ and $M_2$ to be the splitting field of the local factor of $E_1$ (which coincides with that of~$E_2$) at~$p$, contradicting the fact that $M_1$ and $M_2$ are not isomorphic.    
\end{remark}

\begin{table}
\caption{Table of notations}
\label{table:notations}
\begin{tabular}{c|c|c}
Notation & Meaning & First usage \\
\hline
$a$ & Rank of Lie algebra $\fa$ & \S\ref{section: Cartansub} \\
$a_\p$ & Frobenius trace of $A$ at $\p$ & \S\ref{section: introduction} \\
$\overline a_\p$ & Normalized version of $a_\p$ & \S\ref{section: introduction}, \eqref{equation: STPN}  \\
$A$ & Abelian variety over $k$ & \S\ref{section: introduction} \\
$\fa$ & Abelian Lie algebra, factor of $\g$ & \S\ref{section: Cartansub} \\
$\alpha_{\p, j}$ & Reciprocal roots of local L-factor & \S\ref{subsec:motivic}\\
$D$ & Vinogradov function associated to $\Delta, T$ & Proposition~\ref{proposition: vinogradov} \\
$d_\chi$ & Degree of the character $\chi$ & \S\ref{subsec:motivic}\\
$\delta_I$ & Characteristic function of $I$ & \S\ref{section: introduction} \\
$\Delta$ & Cutoff parameter in definition of $D$ & Proposition~\ref{proposition: vinogradov} \\
$\epsilon_{\g}$ & Dependence on $\g$ in Theorem~\ref{theorem: intro1} & \eqref{equation: epsilondefinition} \\
$\epsilon_{\g,\g'}$ & Dependence on $\g, \g'$ in Theorem~\ref{theorem: linniksign} & \eqref{equation: epsilondefinition2} \\
$F$ & Average of $D$ over $\W$ & \eqref{equation: defofF} \\
$g$ & Dimension of $A$ & \S\ref{section: introduction} \\
$\g$ & Lie algebra of $\ST(A)$ & \S\ref{section: Cartansub} \\
$\Gamma_\lambda$ & Representation of $\g$ with highest weight $\lambda$ & \S\ref{section: Lie background} \\
$H$ & Cartan subgroup of $\ST(A)$ & \S\ref{section: Cartansub} \\
$h$ & Rank of Lie algebra $\h$ & \S\ref{section: Lie background} \\
$\h$ & Cartan subalgebra of $\s$ & \S\ref{section: Lie background}\\
$I$ & Subinterval of $[-2g, 2g]$ & \eqref{equation: STPN} \\
$|I|$ & Length of $I$ & \S\ref{section: introduction} \\
$k$ & Number field over which $A$ is defined & \S\ref{section: introduction}\\
$\Li(x)$ & Logarithmic integral & \eqref{equation: STPN} \\
$M$ & Cutoff parameter in weight space & Proof of Theorem \ref{theorem: effectivest} \\
$M_k(x) $ & Extremal primes of norm up to $x$ & \S\ref{section: introduction} \\
$m_\lambda^\mu$ & Weight multiplicity & \S\ref{section: Lie background} \\
$\mu$ & Pushforward of Haar measure on $\ST(A)$ & \S\ref{section: introduction} \\
$N$ & Absolute conductor of $A$ & \S\ref{section: introduction} \\
$\nu_\g$  & Cutoff for $O$ notation & \eqref{equation: nuinterval}\\
$\p$ & Prime ideal of $k$ & \S\ref{section: Lie background} \\
$\varphi$ & Size of set $\Phi^+$ & \eqref{equation: epsilondefinition} \\
$\Phi$ & Root system for $\s$ & \S\ref{section: Lie background} \\
$q$ & Rank of Lie algebra $\g$ ($ = h+a$) & \eqref{equation: epsilondefinition} \\
$\mathcal R$ & Lattice of integral weights of $\s$ & \S\ref{section: Lie background} \\
$S$ & Simple roots of $\Phi$ & \S\ref{section: Lie background} \\
$\s$ & Semisimple factor of $\g$ & \S\ref{section: Lie background}, \S\ref{section: Cartansub} \\
$\ST(A)$ & Sato--Tate group of $A$ & \S\ref{section: introduction}, \S\ref{subsec: Sato-Tate groups} \\
$T$ & Trace map on $\R^q$ & \S\ref{section: Cartansub}, \eqref{equation: mapT} \\
$T_\ell(A)$ & Tate module of $A$ & \S\ref{section: introduction}\\
$V_\ell(A)$ &  $T_\ell(A)\otimes \Q_\ell$ & \S\ref{section: introduction}\\
 $w$ & Element of $\W$ & \S\ref{section: Lie background} \\
$\W$ & Weyl group of $\s$ & \S\ref{section: Lie background} \\
$\omega_j$ & Basis element of fundamental weights & \S\ref{section: Lie background}
\end{tabular}
\end{table}

\end{document}